\documentclass{amsart}
\usepackage{amsmath}
\usepackage{amssymb}
\usepackage{amsfonts}
\usepackage{graphicx}
\usepackage{texdraw}
\usepackage{graphpap}
\usepackage{wasysym}
\usepackage{enumitem}
\usepackage{color}
\usepackage[OT2,T1]{fontenc}
\usepackage{verbatim} 
\usepackage{multirow}
\usepackage{subcaption}
\usepackage[utf8]{inputenc}
\usepackage[english]{babel}
\usepackage[colorlinks=true, citecolor=blue, filecolor=black, linkcolor=black, urlcolor=black]{hyperref}
\usepackage{mathtools}
\usepackage{enumitem}
\usepackage{bbm}

\mathtoolsset{showonlyrefs}

\theoremstyle{plain}
\newtheorem{thm}{Theorem}

\newtheorem{cor}[thm]{Corollary}
\newtheorem{lem}[thm]{Lemma}

\theoremstyle{remark}

\newtheorem*{rem}{Remark}

\numberwithin{equation}{section}

\newcommand{\bfR}{\mathbf{R}}

\newcommand{\bfH}{{\mathbf H}}
\newcommand{\bfK}{{\mathbf K}}

\newcommand{\erfc}{\operatorname{erfc}}

\newcommand{\R}{{\mathbb R}}

\newcommand{\C}{{\mathbb C}}
\newcommand{\E}{{\mathbf E}}

\newcommand{\eps}{{\varepsilon}}

\newcommand{\re}{\operatorname{Re}}

\newcommand{\RN}[1]{%
	\textup{\uppercase\expandafter{\romannumeral#1}}%
}

\newcommand{\overbar}[1]{\mkern 1.5mu\overline{\mkern-1.5mu#1\mkern-1.5mu}\mkern 1.5mu}

\makeatletter
\@namedef{subjclassname@2020}{\textup{2020} Mathematics Subject Classification}
\makeatother

\begin{document}


\author{Gernot Akemann}

\address{Faculty of Physics, Bielefeld University, P.O. Box 100131, 33501 Bielefeld, Germany}
\email{akemann@physik.uni-bielefeld.de}

\author{Sung-soo Byun}

\address{Department of Mathematical Sciences, Seoul National University, Seoul, 151-747, Republic of Korea}
\email{sungsoobyun@snu.ac.kr}

\author{Nam-Gyu Kang}

\address{School of Mathematics, Korea Institute for Advanced Study, Seoul, 02455, Republic of Korea}
\email{namgyu@kias.re.kr}


\title[A non-Hermitian generalisation of the Marchenko-Pastur 
law]{A non-Hermitian generalisation of the Marchenko-Pastur distribution: \linebreak from the circular law to multi-criticality}

\begin{abstract}
We consider the complex eigenvalues of a Wishart type random matrix model $X=X_1 X_2^*$, where two rectangular complex Ginibre matrices $X_{1,2}$ of size $N\times (N+\nu)$ are correlated through a non-Hermiticity parameter $\tau\in[0,1]$. 
For general $\nu=O(N)$ and $\tau$ we obtain the global limiting density and its support, given by a shifted ellipse. It provides a non-Hermitian generalisation 
of the Marchenko-Pastur distribution, which is recovered at maximal correlation $X_1=X_2$ when $\tau=1$. The square root of the complex Wishart eigenvalues, corresponding to the non-zero complex eigenvalues of the Dirac matrix 
$\mathcal{D}=\begin{pmatrix} 
0 & X_1 \\
X_2^*  & 0 
\end{pmatrix},$
are supported in a domain parametrised by a quartic equation. 
It displays a lemniscate type transition at a critical value $\tau_c,$ where the interior of the spectrum splits into two connected components.
At multi-criticality we obtain the limiting local kernel given by the edge kernel of the Ginibre ensemble in squared variables. For the global statistics we apply Frostman's equilibrium problem to the 2D Coulomb gas, whereas the local statistics follows from a saddle point analysis of the kernel of orthogonal Laguerre polynomials in the complex plane.
\end{abstract}

\thanks{ The authors are grateful to the DFG-NRF International Research Training Group IRTG 2235 supporting the Bielefeld-Seoul graduate exchange programme. 
Furthermore, Gernot Akemann was partially supported 
by the DFG through the grant  CRC 1283 ``Taming uncertainty and profiting from randomness and low regularity in analysis, stochastics and their applications'' and by the Knut and Alice Wallenberg Foundation. 
Nam-Gyu Kang was partially supported by Samsung Science and Technology Foundation (SSTF-BA1401-51) and by a KIAS Individual Grant (MG058103) at Korea Institute for Advanced Study.}

\maketitle
\section{Introduction and Discussion of Main Results}\label{Intro}

In this work we study the complex eigenvalues of the product of two rectangular complex random matrices that are correlated, which is known to form a determinantal point process. One of our goals is to find an interpolation between classical results for random matrices on the global scale, the circular law \cite{girko1985circular} and the Marchenko-Pastur distribution \cite{MR0208649}. Furthermore, we are interested in the local behaviour of correlation functions, in particular at multi-critical points.

The product ensemble of random matrices that we analyse can be seen as a multiplicative version of the elliptic Ginibre ensemble also called Ginibre-Girko ensemble \cite{girko1986elliptic,MR948613},  where the sum of a complex Hermitian and anti-Hermitian random matrix is considered, that are coupled through a non-Hermiticity parameter $\tau$. It allows to interpolate between the circular law for independent matrices and the semi-circle law in the Hermitian limit on a global scale. The local correlations have been shown to be universal in the bulk and at the edge of the spectrum \cite{MR3845296,lee2016fine,MR2722794}.

The random two-matrix model we consider has appeared under the name of \textit{chiral complex Ginibre} or \textit{non-Hermitian Wishart ensemble}. The former name has been used in an application to the Dirac operator spectrum of Quantum Chromodynamics (QCD) with chemical potential \cite{osborn2004universal} (see also \cite{stephanov1996random}), whereas the latter was used as a proposed model for the analysis of time series, e.g., when building a covariance matrices from time-lagged correlation matrices \cite{kanzieper2010non}. Also in \cite{kanzieper2010non,osborn2004universal}, special cases of global and local statistics were analysed. 
In addition to the parameter $\tau$ as in the elliptic Ginibre ensemble, we have a second parameter related to the rectangularity or zero eigenvalues of the random matrices. The singular value statistics of this ensemble has been analysed as well \cite{MR3509011}. 
We also refer to \cite{MR2881072} for an extension of the Ginibre ensemble to include zero eigenvalues.

Let us be more precise now in introducing our model. For given non-negative integers $N,\nu$, let $P$ and $Q$ be $N \times (N+\nu)$ random matrices with independent complex Gaussian entries of mean $0$ and variance $1/(4N)$. These are the building blocks of the two correlated random matrices
\begin{equation}
\label{X12def}
X_1=\sqrt{1+\tau} \, P+\sqrt{1-\tau} \, Q, \quad X_2=\sqrt{1+\tau} \, P-\sqrt{1-\tau} \, Q. 
\end{equation}
Here, $\tau \in [0,1]$ is a \textit{non-Hermiticity} parameter, and we use the conventions of \cite{akemann2010interpolation}. Only for $\tau=0$, $X_1$ and $X_2$ are again uncorrelated Gaussian random matrices. In the other extremal case $\tau=1$, the two matrices become perfectly correlated, $X_1=X_2$. 
In the following we consider the complex eigenvalues of the non-Hermitian Wishart matrix $X$, given by the product of the two:
\begin{equation}
X:=X_1 X_2^*.
\label{Wishart}
\end{equation}
Throughout this article, we shall use the notation
\begin{equation} \label{nu alpha}
\alpha_N:=\frac{\nu}{N}, \quad \lim_{N \to \infty} \alpha_N=\alpha \in [0,\infty).
\end{equation}
The empirical measure $\mu_N$ associated with $X$ is given by 
\begin{equation}
\widehat{\mu}_N:=\frac{1}{N} \sum_{j=1}^N  \delta_{ \widehat{\zeta}_j },
\end{equation}
where $\widehat{\boldsymbol{\zeta}}=\{ \widehat{\zeta}_j \}_{j=1}^N$ are the $N$  complex eigenvalues of $X$, that is the solutions of the characteristic equation 
$0=\det[\widehat{\zeta}-X]$.

It is well known \cite{MR2736204, MR2787469,gotze2010asymptotic} that for the product of $M$ independent complex Gaussian matrices, that is in our case at $M=2$ and at $\tau= \alpha_N = 0$ (see \cite{MR2574103} for an earlier work), the limiting distribution is given by
\begin{equation} \label{product}
d\widehat{\mu}(\zeta)=\frac{1}{ M|\zeta|^{2-2/M}  }\cdot \mathbbm{1}_{ \mathbb{D}}(\zeta) \, dA(\zeta)
\end{equation}
on the unit disc $\mathbb{D}$, where  $dA(\zeta)=d^2\zeta/\pi$ is the two-dimensional (2D) Lebesgue measure divided by $\pi$. On the other hand, in the Hermitian limit for general $\alpha\geq0$, the matrix $X$ becomes positive definite and we have the classical Wishart ensemble. Therefore the limiting spectral distribution follows the Marchenko-Pastur law \cite{MR0208649}
\begin{equation} \label{MP}
\frac{1}{2\pi}\frac{\sqrt{(\lambda_{+}-x) (x-\lambda_{-})  }}{x}\cdot  \mathbbm{1}_{ [\lambda_{-},\lambda_{+} ] }(x), \quad   \lambda_{\pm}:=(\sqrt{\alpha+1}\pm 1)^2.
\end{equation} 
Fig.~\ref{Fig. NWisart} shows some random samplings of eigenvalues of $X$ interpolating these two situations,  with different values $\tau \in [0,1]$ at $\alpha_N=0$ and $\alpha_N=1$.

\begin{figure}[t]
	
	\begin{subfigure}{0.24\textwidth}
		\begin{center}	
			\includegraphics[width=1.046in,height=0.8in]{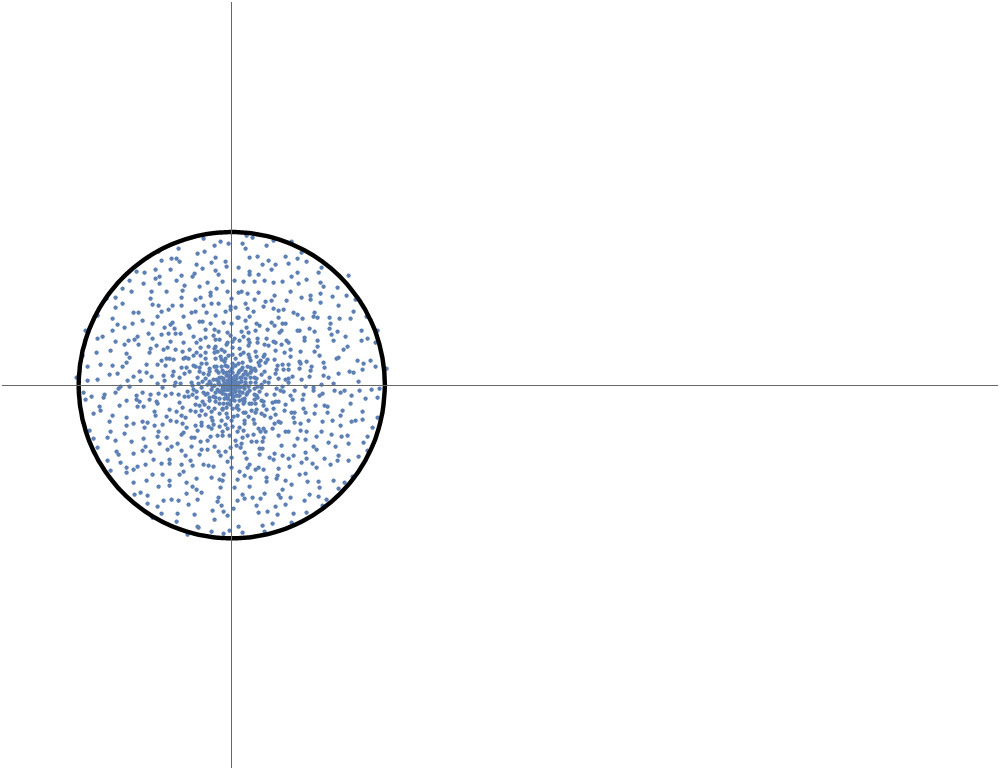}
		\end{center}
		\subcaption{$\tau=0$}
	\end{subfigure}	
    	\begin{subfigure}{0.24\textwidth}
    	\begin{center}	
    		\includegraphics[width=1.046in,height=0.8in]{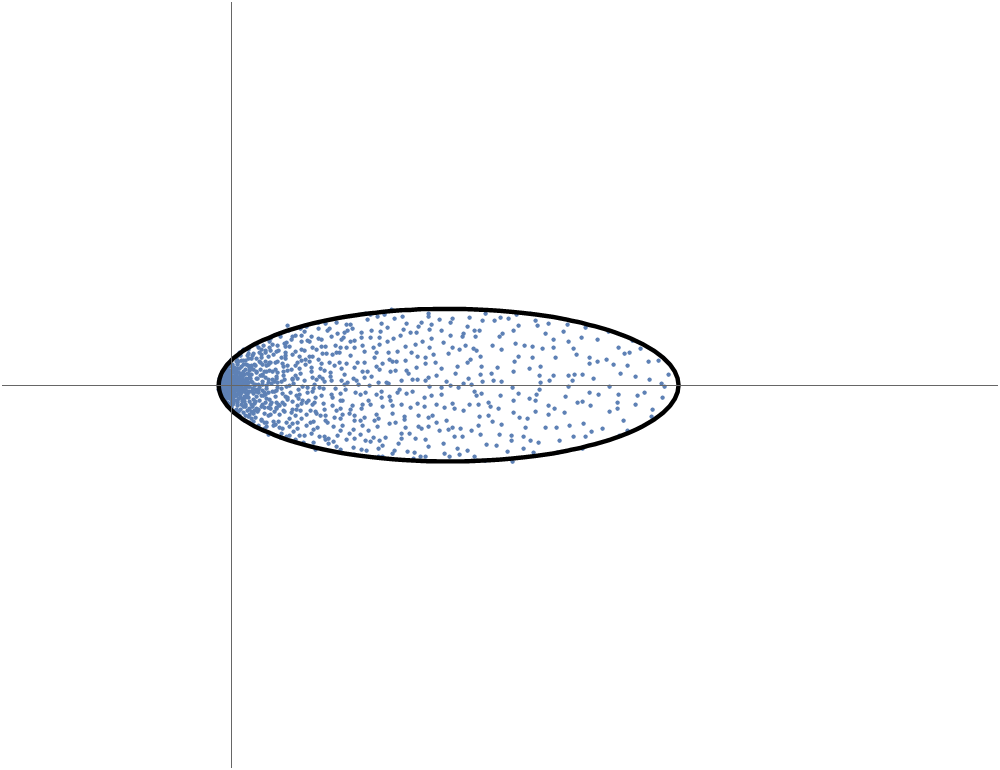}
    	\end{center}
    	\subcaption{$\tau=1/\sqrt{2}$}
    \end{subfigure}	
    \begin{subfigure}{0.24\textwidth}
    	\begin{center}	
    		\includegraphics[width=1.046in,height=0.8in]{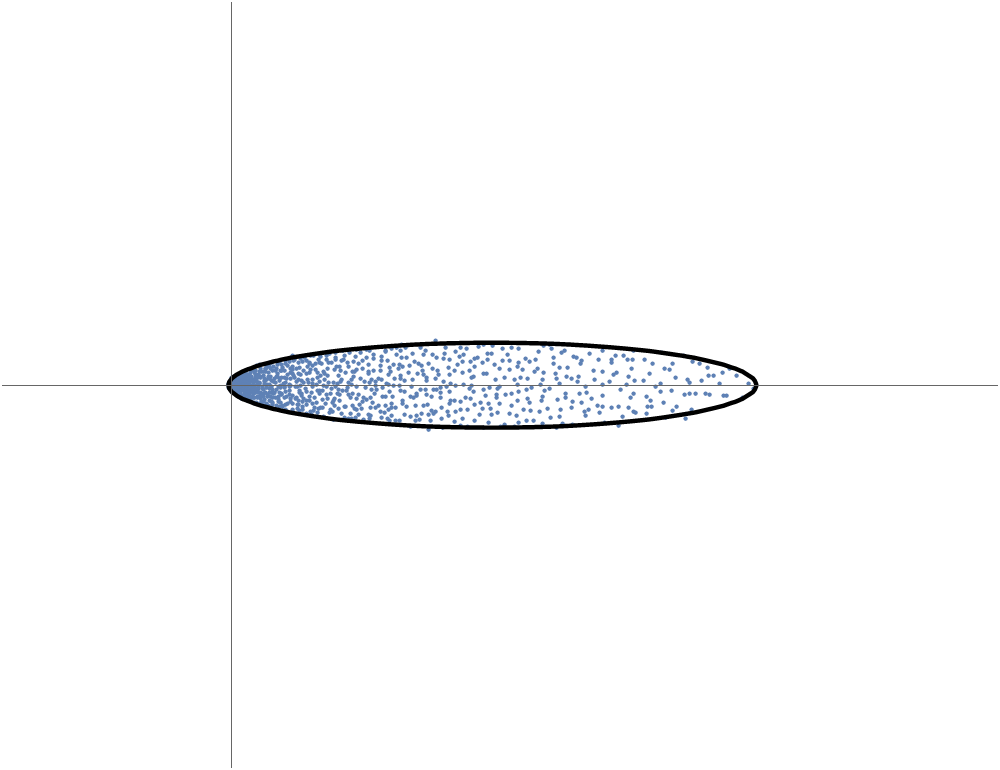}
    	\end{center}
    	\subcaption{$\tau=0.85$}
    \end{subfigure}	
    \begin{subfigure}{0.24\textwidth}
    	\begin{center}	
    		\includegraphics[width=0.965in,height=0.8in]{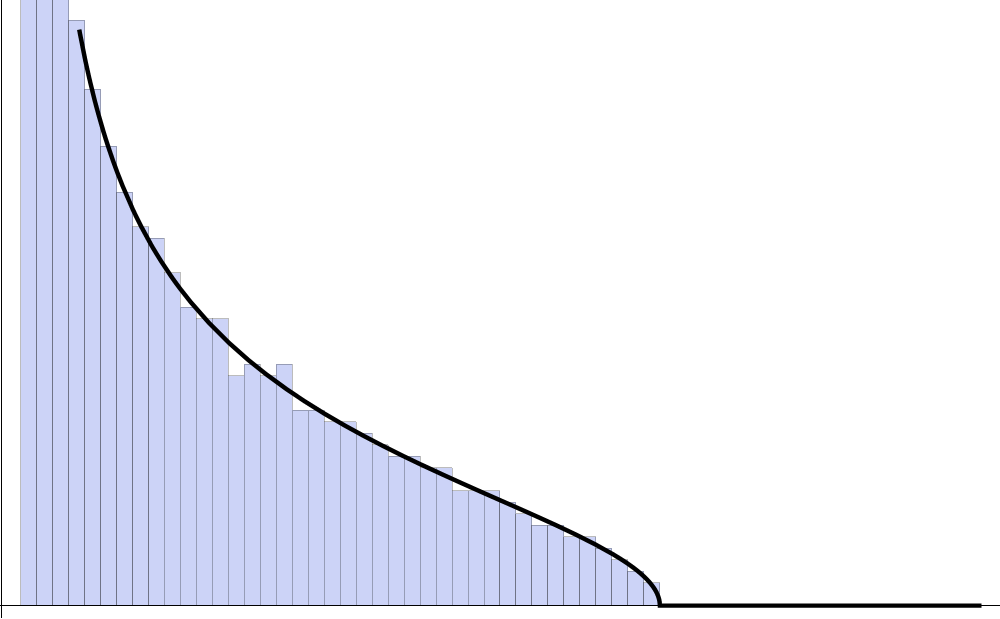}
    	\end{center}
    	\subcaption{$\tau=1 $}
    \end{subfigure}	
    
    \vspace{0.5em}
    
	\begin{subfigure}{0.24\textwidth}
		\begin{center}	
			\includegraphics[width=1.046in,height=0.8in]{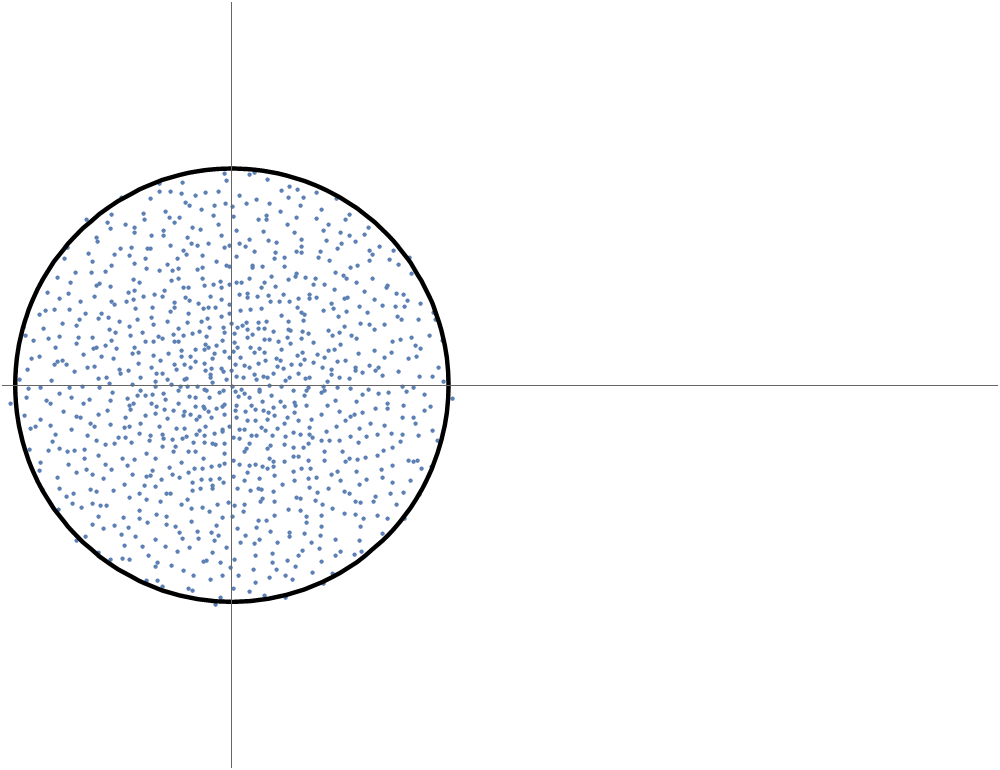}
		\end{center}
		\subcaption{$\tau=0$}
	\end{subfigure}		
	\begin{subfigure}{0.24\textwidth}
		\begin{center}	
			\includegraphics[width=1.046in,height=0.8in]{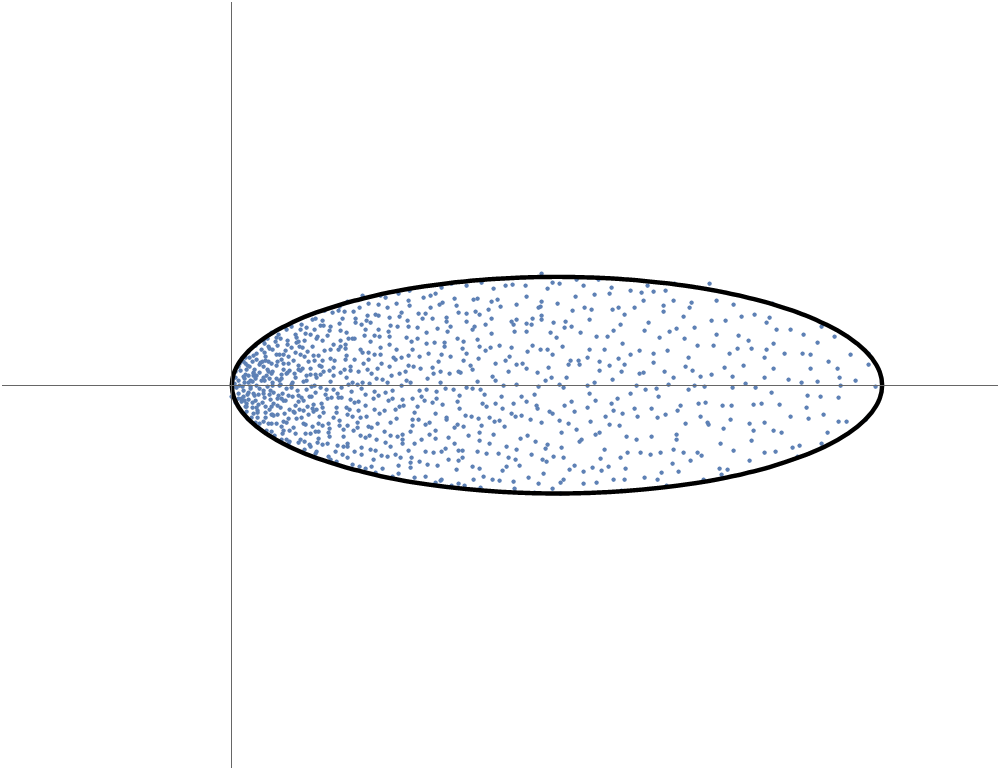}
		\end{center}
		\subcaption{$\tau=1/\sqrt{2}$}
	\end{subfigure}	
	\begin{subfigure}{0.24\textwidth}
		\begin{center}	
			\includegraphics[width=1.046in,height=0.8in]{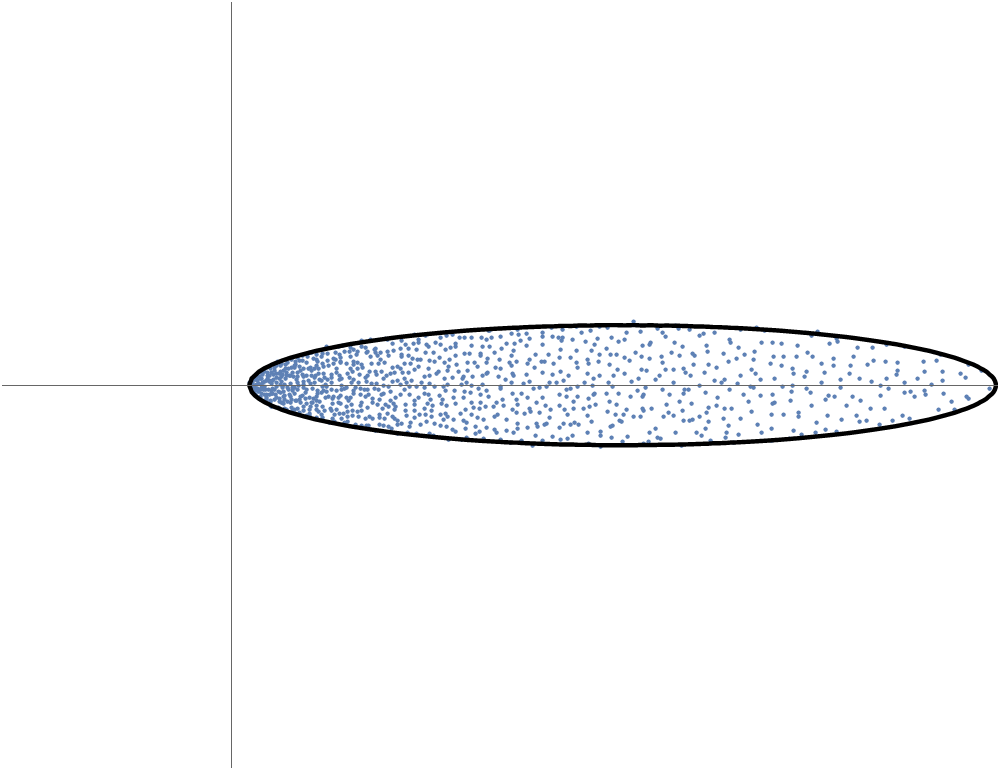}
		\end{center}
		\subcaption{$\tau=0.85 $}
	\end{subfigure}		
	\begin{subfigure}{0.24\textwidth}
		\begin{center}	
			\includegraphics[width=0.965in,height=0.8in]{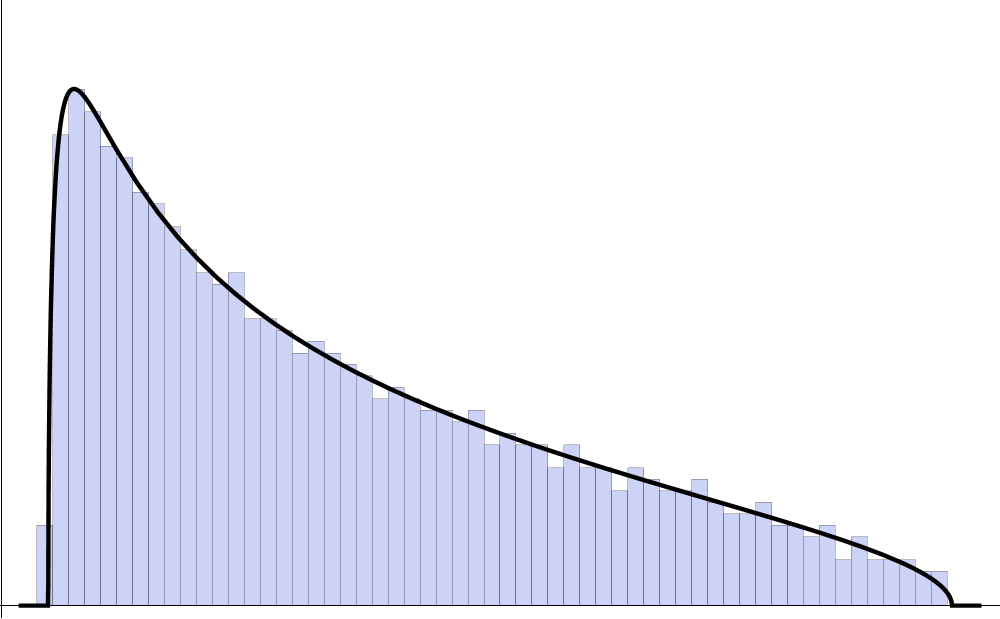}
		\end{center}
		\subcaption{$\tau=1$}
	\end{subfigure}		
	\caption{ The eigenvalues of $X$, where $N=1000.$ Here $\nu=0$ for the figures (A)~--~(D) in the top row and $\nu=N$ for (E)~--~(H) in the bottom row. In particular, when $\tau=1$, the figures (D) and (H) display histograms of the positive eigenvalues whose distributions follow the Marchenko-Pastur law \eqref{MP} with $\alpha=0$ and $\alpha=1$, respectively. 
	} \label{Fig. NWisart}
\end{figure} 

In our first main result Theorem~\ref{Thm_NWishart}  we derive the limiting global spectral distribution of $X$ for arbitrary $\alpha$ and $\tau$, 
which includes the above-mentioned limiting cases. In that sense it
provides a \textit{non-Hermitian generalisation of the Marchenko-Pastur law}.
In particular, we show that the droplet (support of the spectrum) is enclosed by an ellipse with foci $c_{\pm}:=\tau \lambda_{\pm}.$ 

We say that the empirical measure $\widehat{\mu}_N$ weakly converges to $\widehat{\mu}$ if for each bounded continuous function $f$, 
\begin{equation}
\frac{1}{N} \widehat{\E}_N\Big[ f(\widehat{\zeta}_1)+\cdots +f(\widehat{\zeta}_N) \Big] \to \int f \, d \widehat{\mu}, 
\end{equation}
where $\widehat{\E}_N$ is the expectation with respect to the underlying Gibbs measure.

\begin{thm}\label{Thm_NWishart}
As $N \to \infty$, the empirical measure $\widehat{\mu}_N$ weakly converges to $\widehat{\mu}$, where 
\begin{equation} \label{NWishart_density}
d\widehat{\mu}(\zeta):=\frac{1}{1-\tau^2}  \frac{1}{ \sqrt{    4 |\zeta|^2+ (1-\tau^2 )^2 \alpha ^2   }  }\cdot \mathbbm{1}_{ \widehat{S}_\alpha }(\zeta) \, dA(\zeta).
\end{equation}
 Here, the support $ \widehat{S}_\alpha $ of the spectrum is given by
\begin{equation} \label{hat S alpha}
 \widehat{S}_\alpha := \Big\{ \zeta=x+iy:  \Big( \frac{x-\tau (2+\alpha)  }
{ (1+\tau^2) \sqrt{1+\alpha} }   \Big)^2+\Big( \frac{y}{ (1-\tau^2)  \sqrt{1+\alpha} }  \Big) ^2 \le 1 \Big\}.
\end{equation}
\end{thm}
Note that the origin is on the edge of the spectrum \eqref{hat S alpha} if and only if $\tau$ is given by the critical value
\begin{equation}
\tau_c:=\frac{1}{\sqrt{1+\alpha}}.
\end{equation} 
For $\tau<\tau_c$ the origin is inside and for $\tau>\tau_c$ outside the support.

We remark that for $\tau=\alpha_N=0$, Theorem~\ref{Thm_NWishart} reduces to the known density of the product of two independent Gaussian matrices as already mentioned, see \cite{MR2736204,MR2787469,MR2574103,gotze2010asymptotic,MR2861673}. Moreover, for $\tau=0$ with general $\alpha \ge 0$, the spectrum of $X$ was studied by Kanzieper and Singh in \cite{kanzieper2010non}. 
Apparently the limiting global spectral distribution in Theorem~\ref{Thm_NWishart} seems  to be universal. For instance, in \cite{benet2014spectral}, the same distribution was shown to follow  from loop equations for the product of two real rectangular random matrices correlated by a diagonal matrix. For the product of $M$ elliptic Ginibre matrices it was shown in \cite{MR3357969} that the law \eqref{product} is universal. 

In Section \ref{Section_global}, Theorem~\ref{Thm_NWishart} is derived using the fact that the empirical measure of 2D Coulomb gases concentrates on Frostman's equilibrium measure (see e.g., \cite{HM13,MR1487983}), which reduces the proof of Theorem \ref{Thm_NWishart} to solving the associated equilibrium problem. 
For recent developments on concentration for Coulomb gases, 
see \cite{MR3820329,MR3735628} and references therein. We also refer to a recent work \cite{del2019equilibrium} of Criado del Rey and Kuijlaars on the equilibrium problem associated with a certain ellipse. 
\\


For the investigation of the local statistics, in particular when comparing to data in applications, it is important to unfold the spectrum, that is to map the global density to a constant at the point, around which the fluctuations are measured. This applies equally to one- and to two-dimensional spectra. For instance, it is known that for the product of $M$ independent random matrices at fixed $\nu$ such a map is provided by taking the $M$-th root, in our case $M=2$ the square root. This  maps the density \eqref{product} to the circular law. After this map the local statistics equal those of a single Ginibre matrix in the bulk (away from the origin) and at the edge \cite{MR2993423}, cf. \cite{liu2019phase} for a recent rigorous derivation of this fact. Likewise, taking the square root in the Wishart ensemble at $\tau=1$ and $\alpha=0$ maps the local Bessel law at the origin to fluctuate around a locally constant density, sometimes also called the quarter circle law, see \eqref{MP square} below. 

This feature can be conveniently described in an equivalent formulation of the Wishart ensemble  \eqref{Wishart} of the product of two correlated matrices, also called \textit{chiral} version of $X$. More precisely, let
\begin{equation}
\mathcal{D}:=\begin{pmatrix} 
0 & X_1 \\
X_2^*  & 0 
\end{pmatrix}
\end{equation}
be the $(2N+\nu) \times (2N+\nu)$ random \textit{Dirac matrix}. In particular, for the extremal case $\tau=1$, this non-Hermitian two-matrix ensemble $\mathcal{D}$ reduces to the standard Hermitian chiral Gaussian unitary ensemble with $X_1=X_2$, see \cite{JacEdward}. For the details of model $\mathcal{D}$ as well as its physical applications, we refer to \cite[Chapter 32]{akemann2011oxford}, \cite[Section 15.11]{forrester2010log}, \cite{osborn2004universal} and references therein. 

The spectrum of $\mathcal{D}$ consists of the deterministic eigenvalue $0$ with multiplicity $\nu$ and $2N$ complex eigenvalues $\{ \pm \zeta_j \}_{j=1}^N$, which always come in pairs, whence the name chiral. The complex eigenvalues of $\mathcal{D}$ are the solutions of the characteristic equation $0=\det[\zeta-\mathcal{D}]$. For the non-zero eigenvalues this immediately leads to the following relation between the Wishart and Dirac matrix eigenvalues:
\begin{equation}
\label{map}
\zeta_j^2=\widehat{\zeta}_j, \quad j=1,\cdots, N.
\end{equation}
We remark that unlike \cite{akemann2010interpolation}, where the $\zeta_j$ are considered in the complex half-plane to avoid a double covering, we consider both sets of eigenvalues as elements of the full complex plane.
Before discussing the local statistics let us draw some consequences from this mapping for the global statistics of the Dirac matrix $\mathcal{D}$, which at first sight seems trivial.

It is known in the physics literature \cite{MR2806770} that for any fixed $\nu$, (in general, for $\nu=o(N)$) the eigenvalue system $\boldsymbol{\zeta}=\{ \zeta_j \}_{j=1}^N$ tends to be uniformly distributed on the domain given by the ellipse
\begin{equation} \label{S alpha 2 0}
S_0:= \Big\{ \zeta=x+iy: \Big( \frac{x}{1+\tau} \Big)^2+\Big( \frac{y}{1-\tau} \Big)^2 \le 1 \Big\},
\end{equation}
as $N \to \infty$. However, the appearance of the ellipse \eqref{S alpha 2 0} was not
rigorously proved. Moreover, when the parameter $\nu$ is proportional to $N$, the limiting global spectrum in the large-$N$ limit has not been considered. See Fig.~\ref{Spectrum nuN} for samplings of the eigenvalues of the Dirac matrix $\mathcal{D}$ for some values in between $0\leq \tau \leq 1$ at $\alpha_N=0$ and $\alpha_N=1$.

\begin{figure}[t]
	\begin{subfigure}{0.19\textwidth}
		\begin{center}	
			\includegraphics[width=0.8in,height=0.8in]{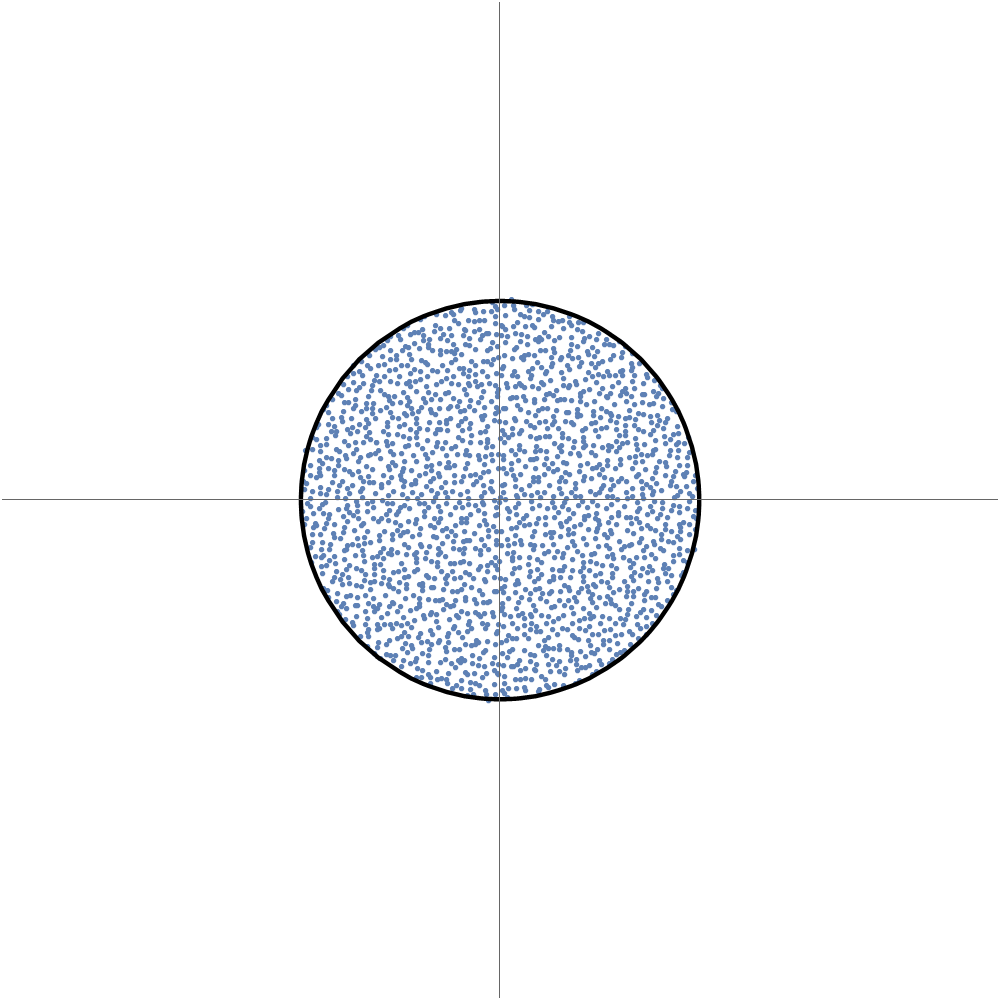}
		\end{center}
		\subcaption{$\tau=0$}
	\end{subfigure}	
	\begin{subfigure}{0.19\textwidth}
		\begin{center}	
			\includegraphics[width=0.8in,height=0.8in]{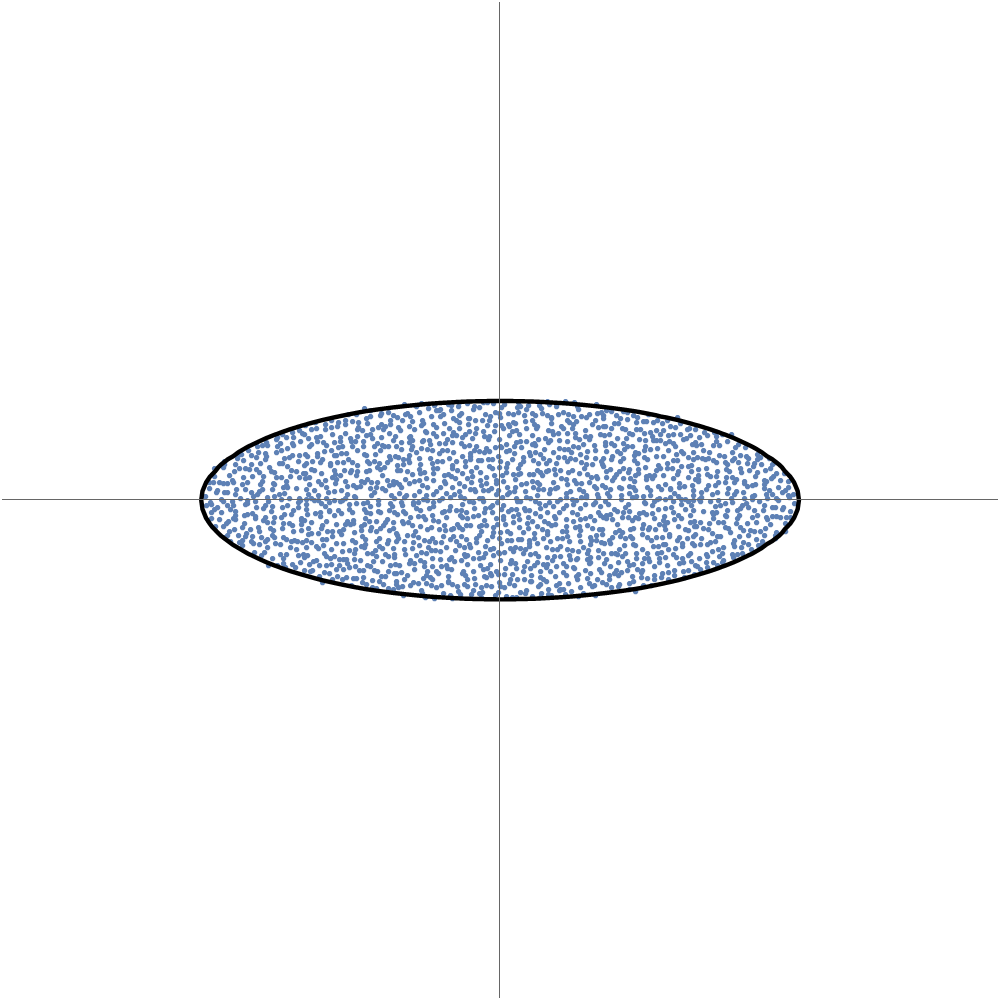}
		\end{center}
		\subcaption{$\tau=0.5$}
	\end{subfigure}	
	\begin{subfigure}[h]{0.19\textwidth}
		\begin{center}
			\includegraphics[width=0.8in,height=0.8in]{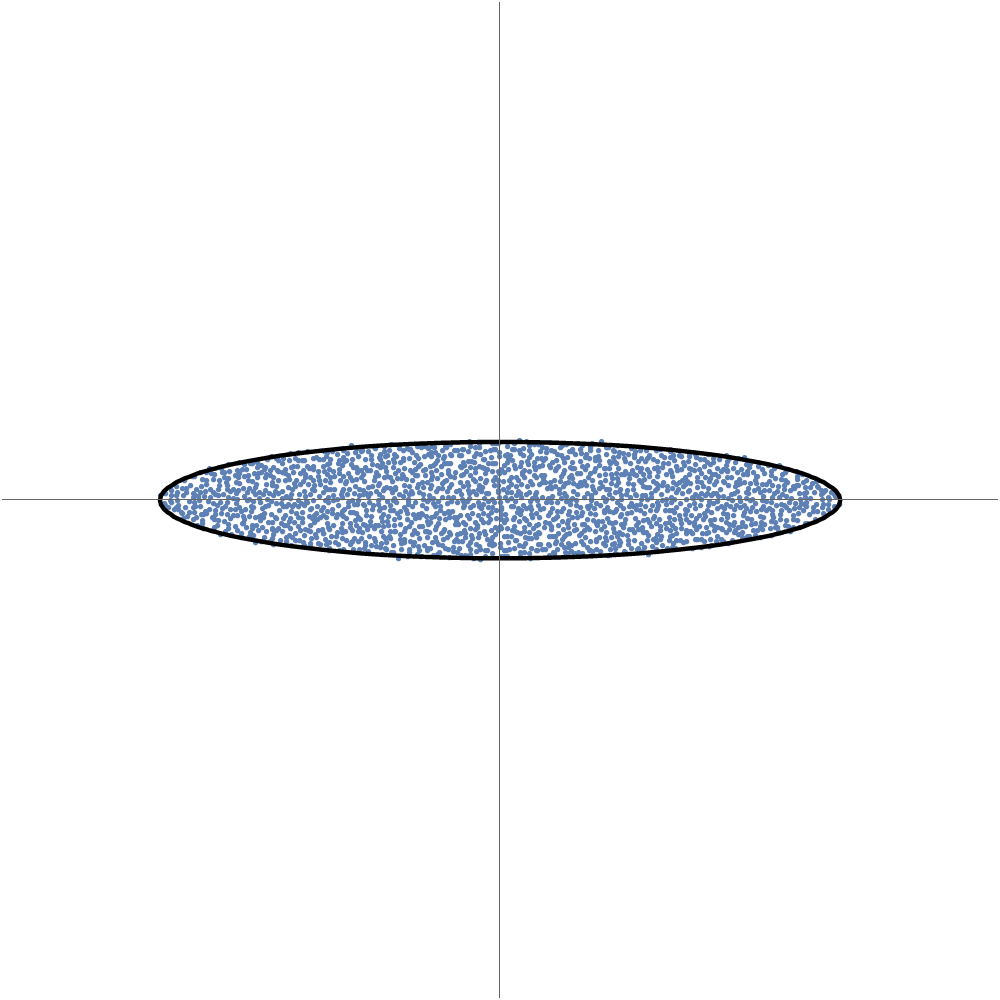}
		\end{center} \subcaption{$\tau=1/\sqrt{2}$}
	\end{subfigure}
	\begin{subfigure}[h]{0.19\textwidth}
		\begin{center}
			\includegraphics[width=0.8in,height=0.8in]{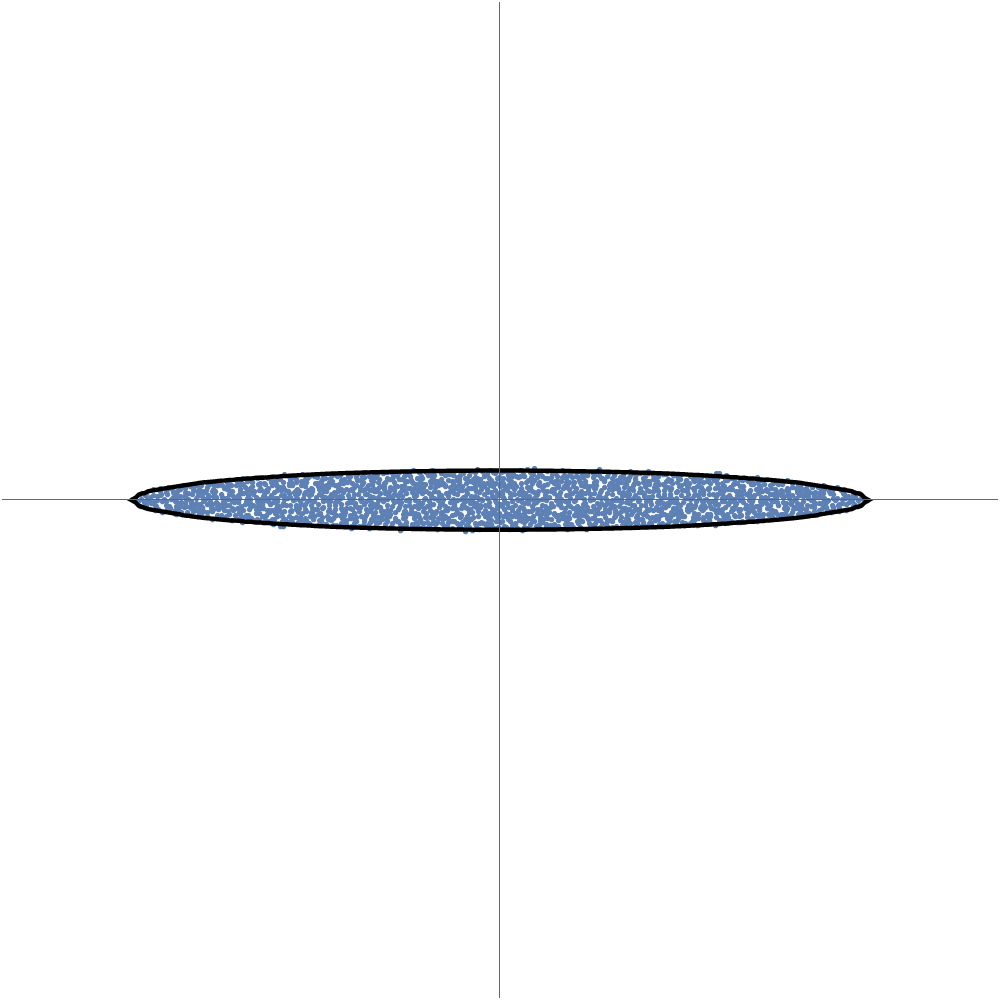}
		\end{center} \subcaption{$\tau=0.85$}
	\end{subfigure}	
	\begin{subfigure}[h]{0.19\textwidth}
		\begin{center}
			\includegraphics[width=0.8in,height=0.8in]{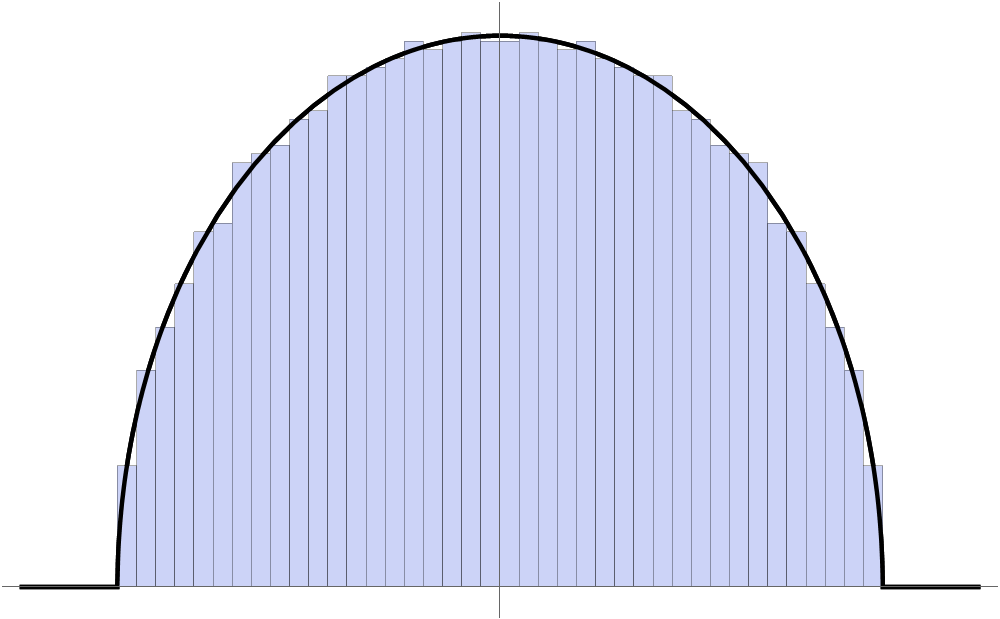}
		\end{center} \subcaption{$\tau=1$}
	\end{subfigure}	
	
	 \vspace{0.5em}
	
	\begin{subfigure}{0.19\textwidth}
		\begin{center}	
			\includegraphics[width=0.8in,height=0.8in]{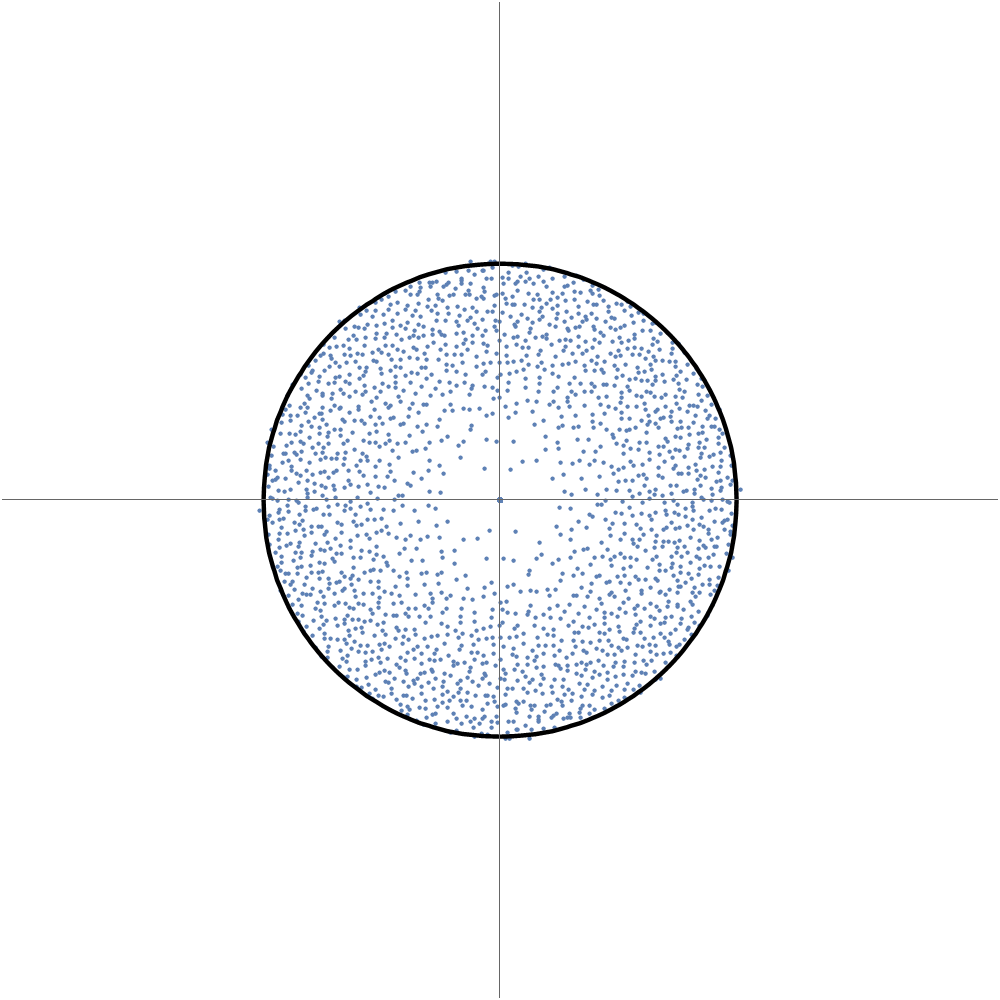}
		\end{center}
		\subcaption{$\tau=0$}
	\end{subfigure}	
	\begin{subfigure}{0.19\textwidth}
		\begin{center}	
			\includegraphics[width=0.8in,height=0.8in]{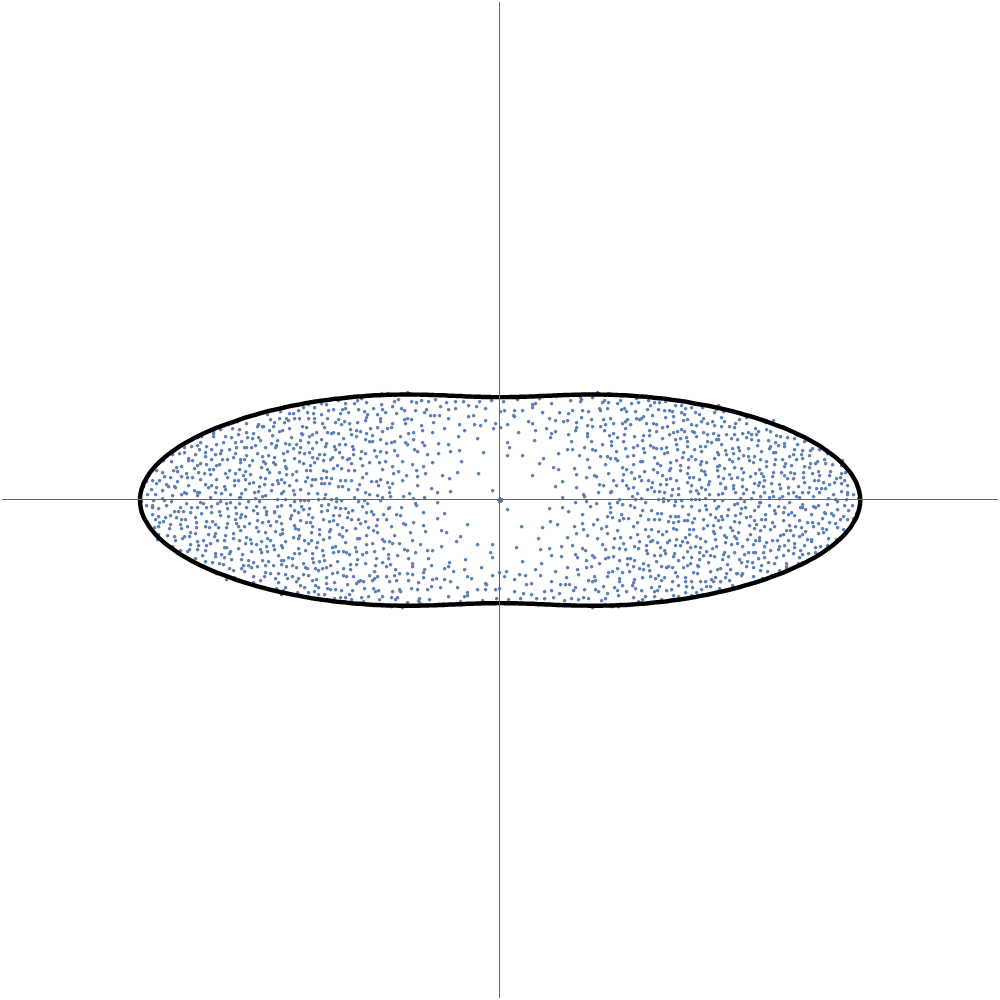}
		\end{center}
		\subcaption{$\tau=0.5$}
	\end{subfigure}	
	\begin{subfigure}[h]{0.19\textwidth}
		\begin{center}
			\includegraphics[width=0.8in,height=0.8in]{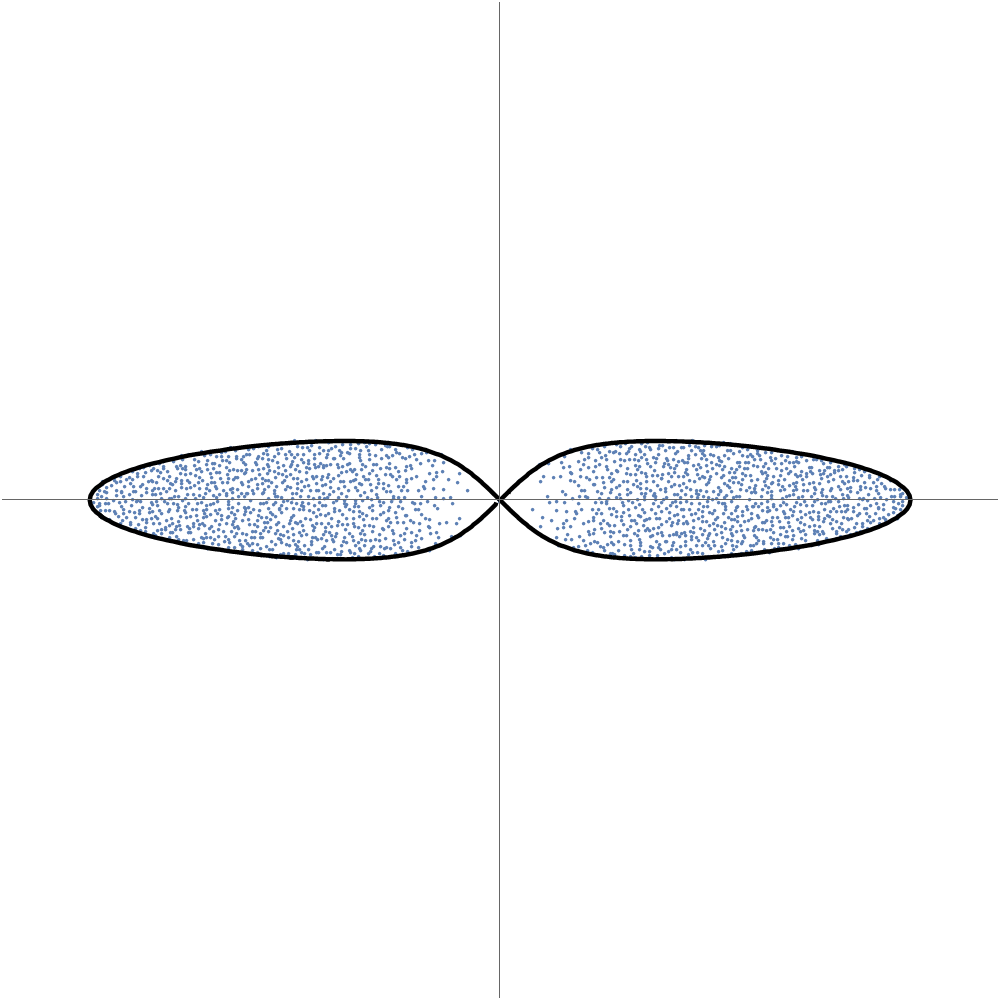}
		\end{center} \subcaption{$\tau=1/\sqrt{2}$}
	\end{subfigure}
	\begin{subfigure}[h]{0.19\textwidth}
		\begin{center}
			\includegraphics[width=0.8in,height=0.8in]{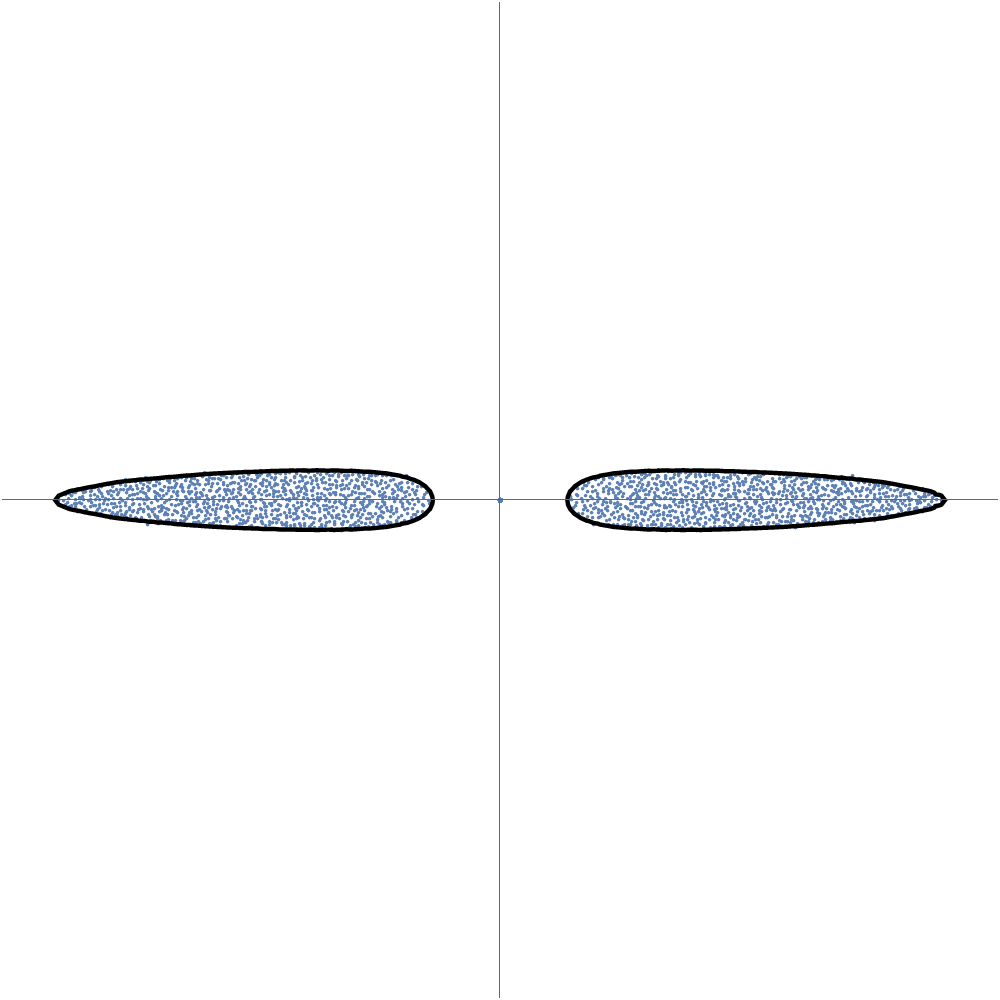}
		\end{center} \subcaption{$\tau=0.85$}
	\end{subfigure}	
	\begin{subfigure}[h]{0.19\textwidth}
		\begin{center}
			\includegraphics[width=0.8in,height=0.8in]{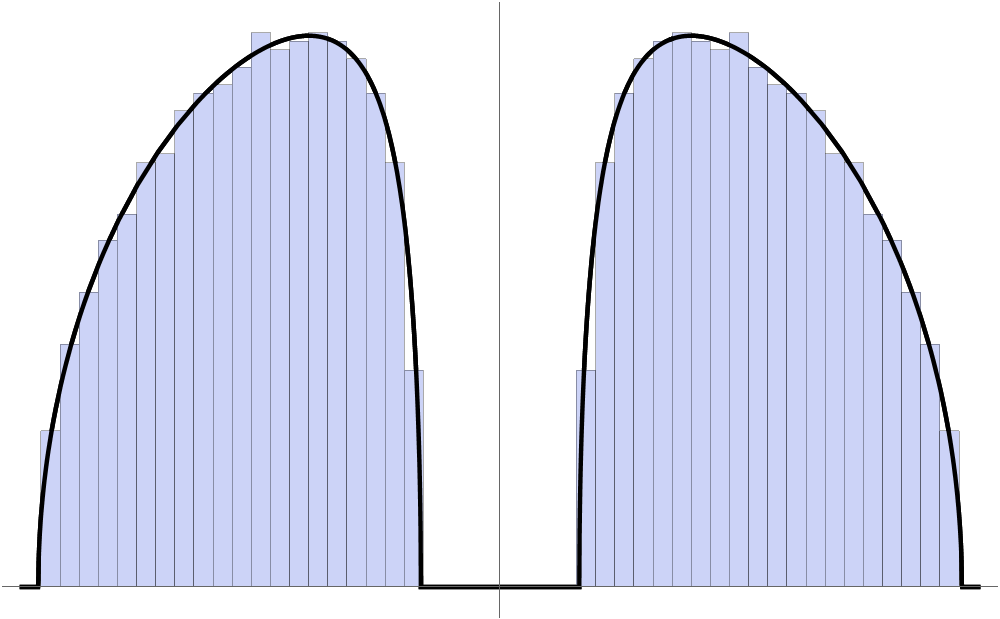}
		\end{center} \subcaption{$\tau=1$}
	\end{subfigure}	
	\caption{ The spectrum of $\mathcal{D}$ (without displaying the zero-modes) where $N=1000$. Here, $\alpha_N=0$ for the figures (A)~--~(E) in the top row, with (A) corresponding to the circular law,  and $\alpha_N=1$ for (F)~--~(J) in the bottom row.
    For $\tau=1$, the plots (E) and (J) display histograms of the real eigenvalues whose distributions follow Marchenko-Pastur law of squared variables \eqref{MP square} with $\alpha=0$ and $\alpha=1$, respectively.	
} \label{Spectrum nuN} 
\end{figure} 

As a simple corollary of Theorem~\ref{Thm_NWishart}, we obtain from the map \eqref{map} an explicit formula for the limiting spectral distribution of $\mathcal{D}$.
\begin{cor} \label{Macro}  
	As $N \to \infty$, the empirical measure $\mu_N$ associated with the Dirac matrix $\mathcal{D}$ weakly converges to 
	\begin{equation}
	\mu:=\frac{\alpha}{2+\alpha} \cdot \delta_0 + \frac{2}{2+\alpha} \cdot \mu_0, 
	\end{equation}
	where $\delta_0$ is the Dirac delta at the origin and the probability measure $\mu_0$ is given by 
	\begin{equation}
	d\mu_0(\zeta):=\frac{1}{1-\tau^2} \frac{|\zeta|^2}{\sqrt{ |\zeta|^4+\alpha^2(1-\tau^2)^2/4 }} \cdot  \mathbbm{1}_{S_\alpha}(\zeta) \, dA(\zeta).
	\end{equation}
	Here, the droplet $S_\alpha$ is enclosed by the quartic curve with equation
	\begin{equation} \label{S alpha 2}
	(x^2+y^2)^2+  \frac{16\tau^2}{(1-\tau^2)^2} x^2y^2-2\tau(2+\alpha)(x^2-y^2) =( 1+\alpha-\tau^2 ) (1-(1+\alpha)\tau^2 ).
	\end{equation}
\end{cor} 
A simple proof of Corollary \ref{Macro} consists of inserting $\widehat{\zeta}=\zeta^2$ into the density \eqref{NWishart_density} including the Jacobian, and into the support \eqref{hat S alpha}.

Note that for $\alpha=0$ the limiting density is flat - as required for unfolding - and $\partial S_\alpha$ is given by the ellipse \eqref{S alpha 2 0} as predicted in the physics literature \cite{MR2806770}. This generalises the elliptic law to the ensemble $\mathcal{D}$ for $\alpha=0$.
On the other hand, for $\alpha>0$, one can easily observe from the formula \eqref{S alpha 2} that the topology of the droplet $S_\alpha$ reveals a transition when $\tau$ passes through the critical value $\tau_c$.
More precisely, for $\alpha>0$ and $\tau \in [0,\tau_c)$ the droplet $S_\alpha$ is a simply connected domain, whereas for $\tau \in (\tau_c,1]$ it consists of two connected components, see Fig.~\ref{Spectrum nuN}. 
We emphasise that at the critical regime $\tau=\tau_c$, the droplet $S_\alpha$ is of lemniscate type and refer the reader to  \cite{MR4030288,MR3280250,balogh2015equilibrium,MR3365300,DS20} for some recent studies on the planar ensembles of similar appearance, containing singular boundary points. Such transitions have been observed earlier in the physics literature, see \cite{Haake91,stephanov1996random}. Consequently, the seemingly trivial map \eqref{map} leads to an intricate behaviour for $\alpha>0$, that calls for an investigation of the local statistics in the vicinity of the multi-critical point at the origin for $\tau=\tau_c$. This is the subject of Theorem \ref{Micro} below.

Notice that in the Hermitian regime $\tau=1$, the limiting global spectral distribution follows the Marchenko-Pastur law of squared variables:
\begin{equation} \label{MP square}
\frac{1}{\pi}\frac{\sqrt{(\lambda_{+}-x^2) (x^2-\lambda_{-})  }}{|x|}\cdot  \mathbbm{1}_{  [-\sqrt{\lambda_{+}},-\sqrt{\lambda_{-}} ] \cup [\sqrt{\lambda_{-}},\sqrt{\lambda_{+}} ] }\,(x), 
\end{equation}  
see e.g., \cite[Proposition 3.4.1]{forrester2010log}. In particular for the case $\alpha=0$, with $\lambda_-=0$ and $\lambda_+=2$, it reduces to the semi-circle distribution (or quarter circle when restricted to $\mathbb{R}_+$). 

We remark that similar to the global universality of the circular law \cite{MR1428519, MR2663633,MR2722794} or the elliptic law \cite{MR948613} in general, we expect that the limiting law in Theorem~\ref{Thm_NWishart} universally appears when we replace the complex Gaussian entries of $P$ and $Q$ by general i.i.d. random variables. We emphasise that it is required to keep the correlation between $X_1$ and $X_2$ as in \eqref{X12def} when choosing Wigner matrices for $P$ and $Q$.\\

In the second part of this article we investigate the local statistics of the ensemble $\mathcal{D}$ when the droplet $S_\alpha$ splits into two connected components. Since the case $\alpha=0$ does not reveal such a multi-critical behaviour, let us assume $\alpha>0$  in the sequel and define the rescaled point process $\boldsymbol{z}=\{ z_j \}_{j=1}^N$ at the origin as 
\begin{equation}\label{microlim}
z_j=(N\delta)^{1/4} \cdot \zeta_j, \quad \delta=\frac{1}{(1-\tau^2)^2 \alpha}.  
\end{equation}
Here, the rescaling order is chosen as in \cite{ameur2018random,MR3742808,chau1998structure} so that the mean eigenvalue spacing of the rescaled process $\boldsymbol{z}$ is of order $O(1)$. Recall that the $k$-point correlation function $R_{N,k}(z_1,\cdots,z_k)$ of the system $\boldsymbol{z}$ is given by
$$
R_{N,k}(z_1,\cdots,z_k)
:= \lim_{\eps \downarrow 0} \frac{ \mathbb{P}( \, \exists \text{ at least one particle in } \mathbb{D}(z_j,\eps), j=1,\cdots,k ) }{\eps^{2k}},
$$
where $\mathbb{D}(z_j,\eps)$ is a disc with centre $z_j$ and radius $\eps.$ See also \eqref{ bfR_{N,k} } and \eqref{ R_{N,k} } for a more standard definition of $R_{N,k}$.

Our second main result deals with the scaling limit of the rescaled process $\boldsymbol{z}$. To our knowledge, Theorem~\ref{Micro} below is the first explicit result for the local statistics at the singular boundary of lemniscate type. On the other hand, for the local statistics at a multi-critical point of unitary random Hermitian matrix ensembles, see \cite{MR1949138,MR2254445, MR2434886} and references therein. 
\begin{thm} \label{Micro}
	Given $\alpha > 0$ and $\tau=\tau_c$, for each $k \in \mathbb{N}$, we have
\begin{equation}
\lim_{N \to \infty} R_{N,k}(z_1,\cdots, z_k)=R_{k}(z_1,\cdots, z_k)= \det \Big[ K(z_j, z_l) \Big]_{ 1 \le j,l \le k }
\end{equation}
uniformly for $z_1,\cdots, z_k$ in compact subsets of $\C$, where 
\begin{equation} \label{K}
K(z,w)=|zw| \, e^{ z^2\bar{w}^2 -|z|^4/2-|w|^4/2  } \erfc\Big( -\dfrac{z^2+\bar{w}^2}{\sqrt{2}}  \Big).
\end{equation}
\end{thm} 
An illustration  of the rescaled local density $R_{N,1}$ at finite $N$ is displayed in Figure \ref{R QCD} and a comparison to its asymptotic limit \eqref{K} is given in Figure \ref{R_V conv}, taking cuts in the $x$- and $y$-directions of the 2D density.

In Section~\ref{Section_local}, we present both heuristic arguments for the appearance of the limiting correlation kernel~\eqref{K} and a rigorous proof based on the steepest descent method. In particular we use here that the complex eigenvalues of $\mathcal{D}$ form a determinantal point process with explicitly known kernel.

We notice that the kernel in \eqref{K} coincides with the kernel at the edge of the Ginibre ensemble in terms of squared variables. 
We also find that for $\tau<\tau_c$ the origin is an inner point of the spectrum and displays bulk statistics of the Ginibre ensemble (in squared variables), whereas for $\tau>\tau_c$ the correlations at the origin vanish, see \eqref{K_N V lim} and Theorem~\ref{Micro standard}. 
This perhaps surprising finding can be intuitively explained as follows. In the 
Wishart picture where we consider complex eigenvalues of the product $X$, in Fig.~\ref{Fig. NWisart} 
apparently the point $\tau_c$ is not special at all. Consequently, we expect local bulk Ginibre statistics everywhere inside the support (away from the origin for $\alpha=0$, and after unfolding) and local edge Ginibre statistics everywhere along the boundary of the support. So far this has only been shown for $\tau=\alpha=0$ \cite{MR2993423,liu2019phase}.

\begin{figure}[t]
	\begin{subfigure}{0.32\textwidth}
		\begin{center}	
			\includegraphics[width=1.3in,height=1.05in]{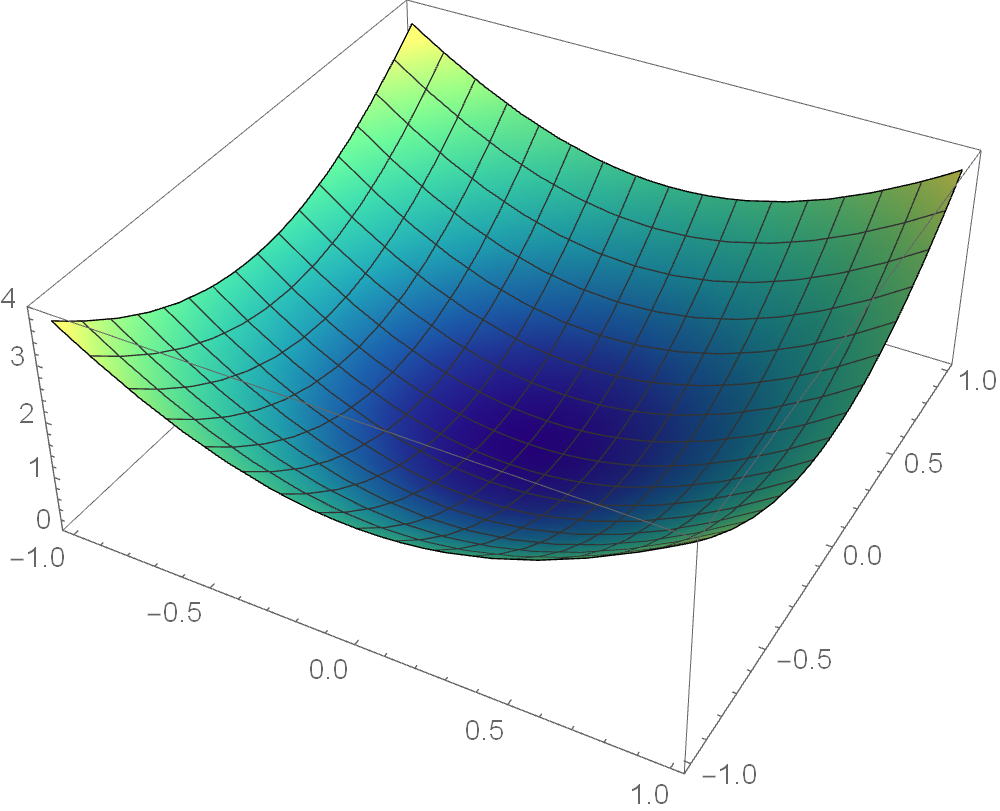}
		\end{center}
		\subcaption{$\tau=0.5$}
	\end{subfigure}	
	\begin{subfigure}[h]{0.32\textwidth}
		\begin{center}
			\includegraphics[width=1.3in,height=1.05in]{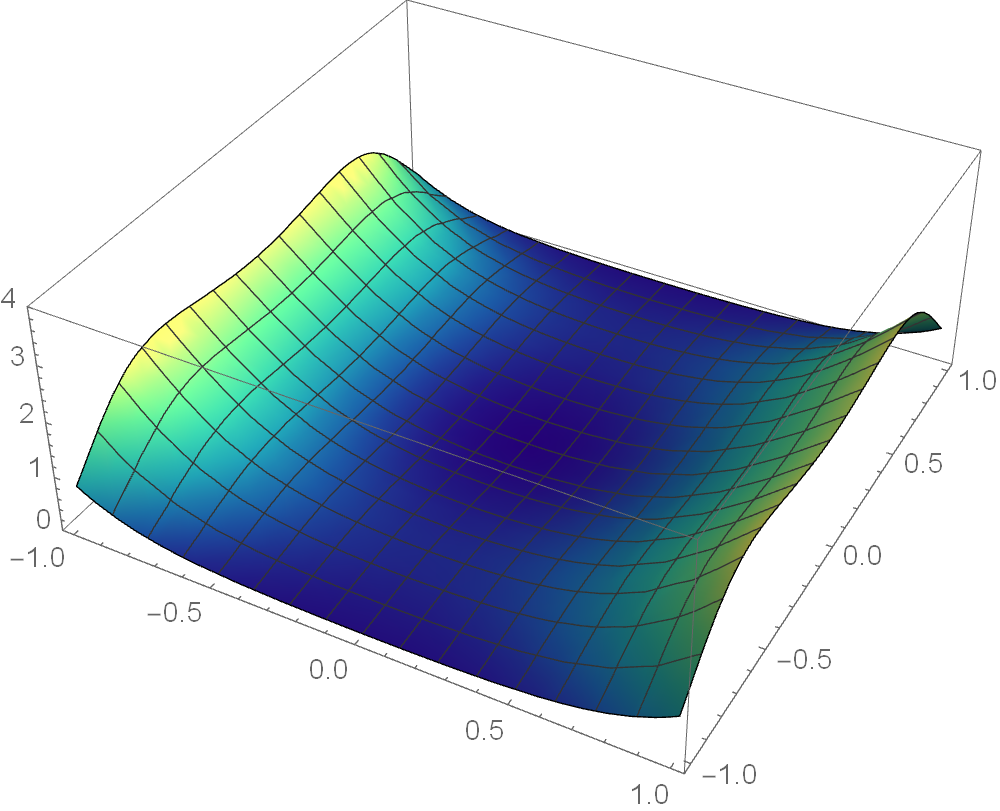}
		\end{center} \subcaption{$\tau=1/\sqrt{2}$}
	\end{subfigure}
	\begin{subfigure}[h]{0.32\textwidth}
		\begin{center}
			\includegraphics[width=1.3in,height=1.05in]{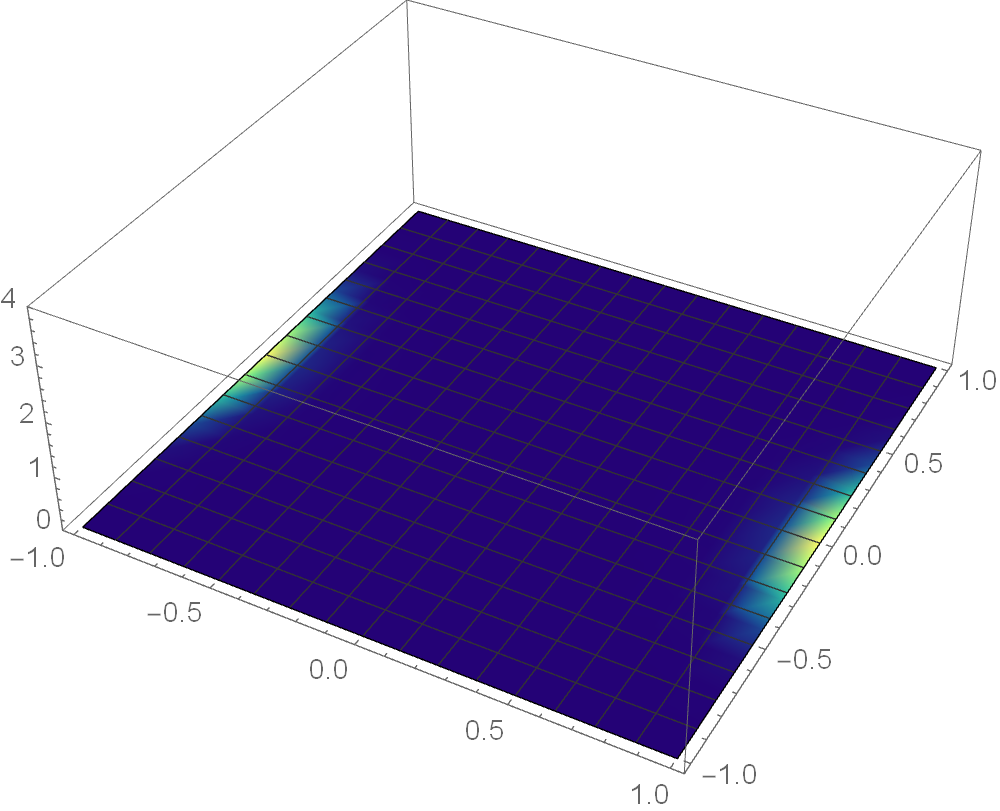}
		\end{center} \subcaption{$\tau=0.85$}
	\end{subfigure}
	\caption{ For illustration we show the graphs of the densities $R_{N,1}$ rescaled for taking the microscopic limit \eqref{microlim}, with $N=100$ and $\alpha_N = 1$. See \eqref{R_{N,1}} for the expression of $R_{N,1}$ in terms of the orthogonal Laguerre polynomials in the complex plane.} \label{R QCD}
\end{figure}

\begin{figure}[t]
	\begin{subfigure}{0.32\textwidth}
		\begin{center}	
			\includegraphics[width=1.22in,height=0.8in]{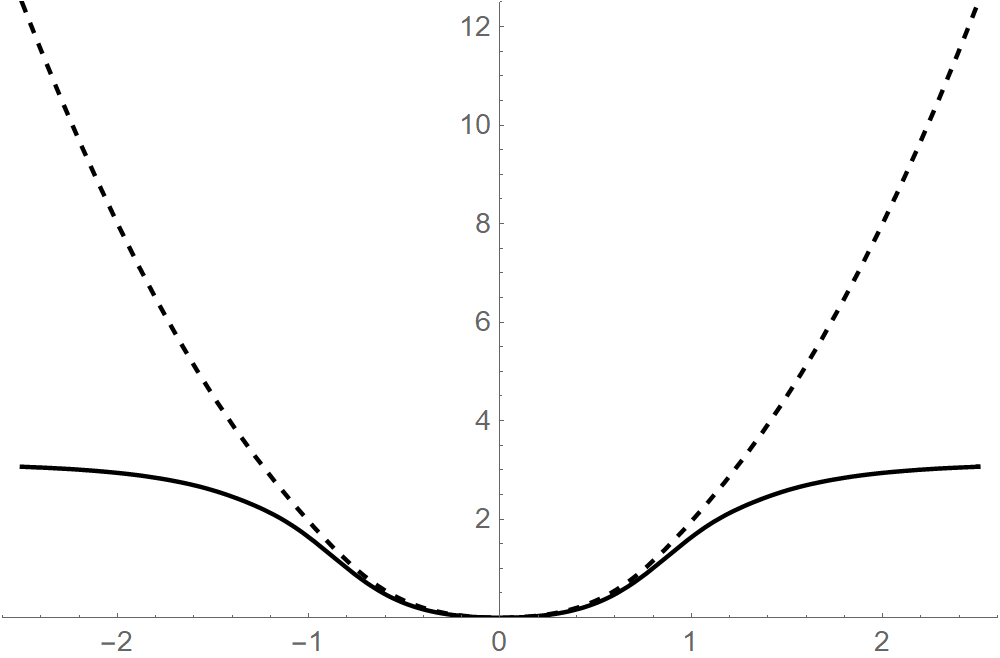}
		\end{center}
		\subcaption{$N=10$}
	\end{subfigure}	
	\begin{subfigure}[h]{0.32\textwidth}
		\begin{center}
			\includegraphics[width=1.22in,height=0.8in]{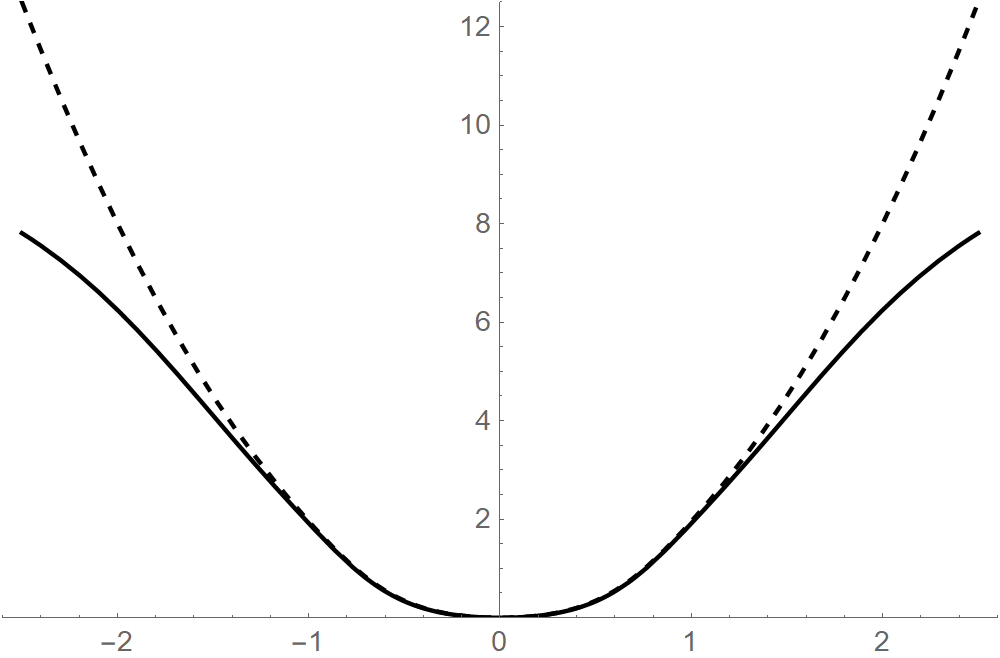}
		\end{center} \subcaption{$N=100$}
	\end{subfigure}
	\begin{subfigure}[h]{0.32\textwidth}
		\begin{center}
			\includegraphics[width=1.22in,height=0.8in]{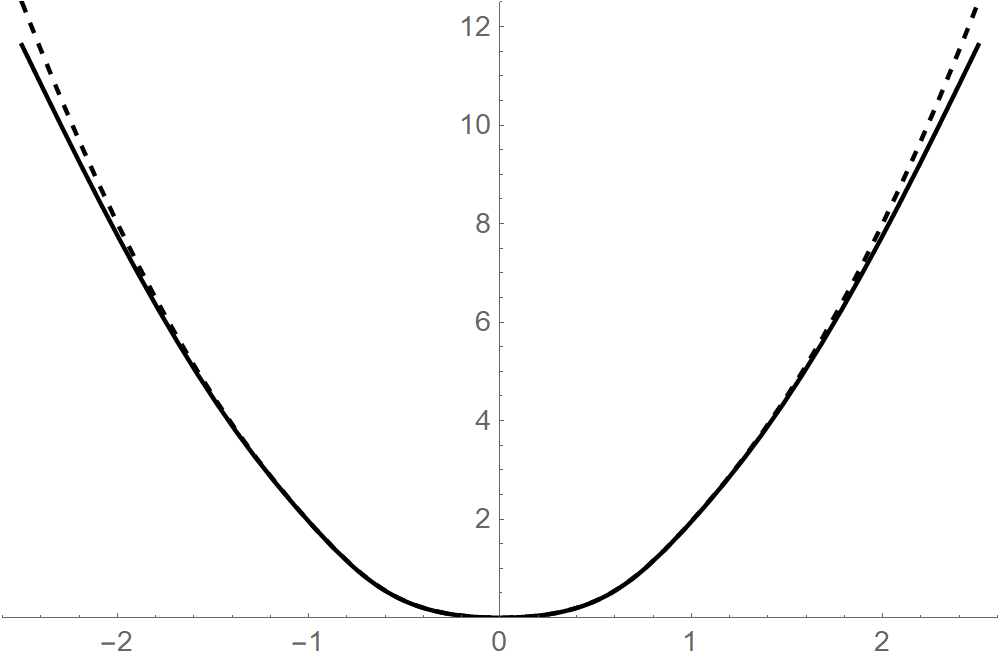}
		\end{center} \subcaption{$N=1000$}
	\end{subfigure}
	
	\vspace{0.5em}
	
		\begin{subfigure}{0.32\textwidth}
		\begin{center}	
			\includegraphics[width=1.22in,height=0.8in]{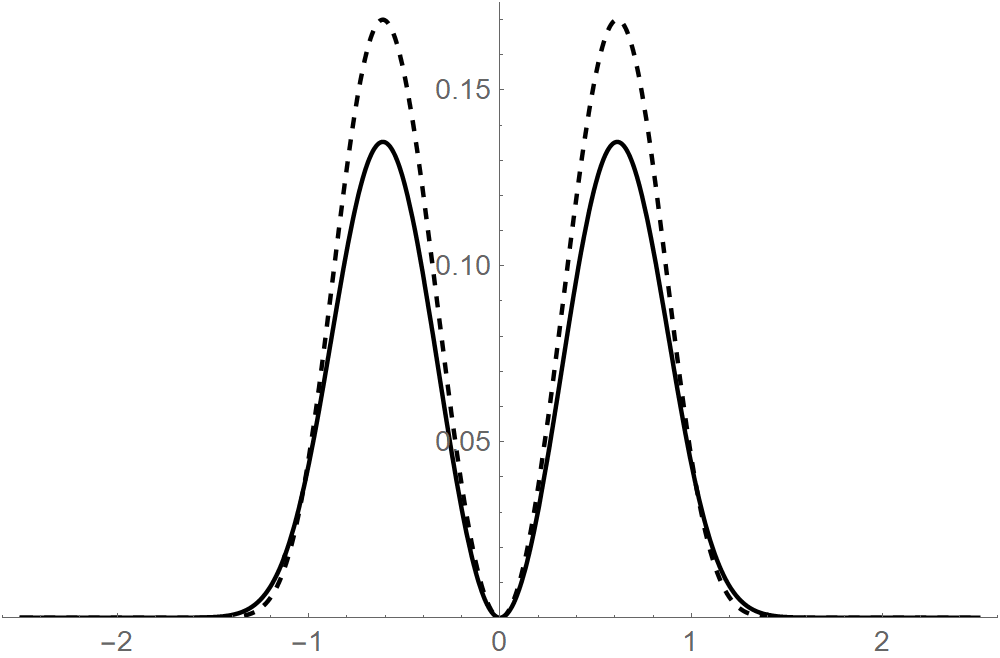}
		\end{center}
		\subcaption{$N=10$}
	\end{subfigure}	
	\begin{subfigure}[h]{0.32\textwidth}
		\begin{center}
			\includegraphics[width=1.22in,height=0.8in]{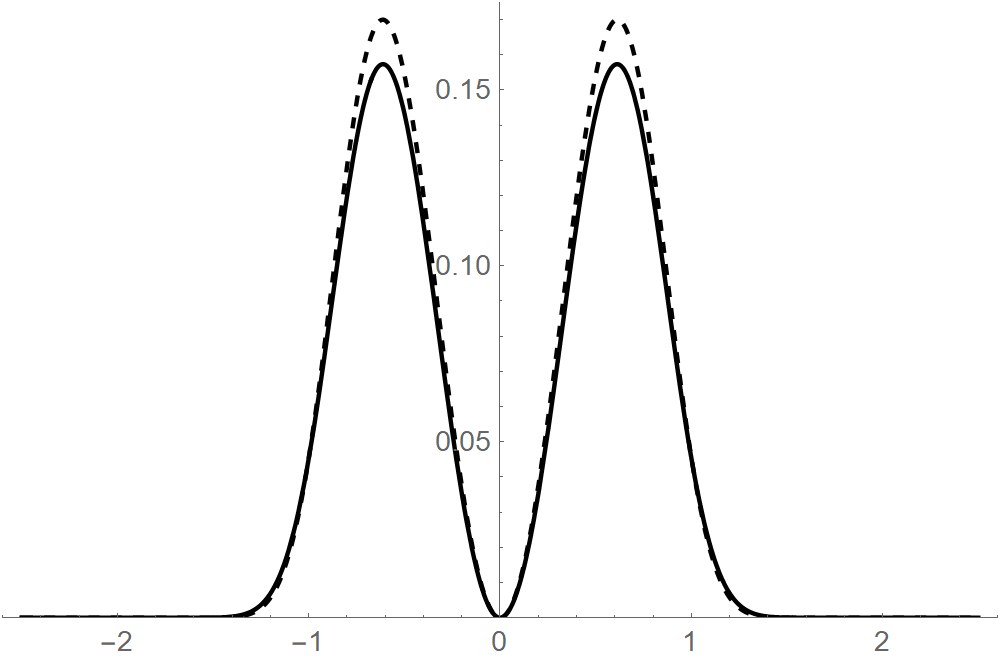}
		\end{center} \subcaption{$N=100$}
	\end{subfigure}
	\begin{subfigure}[h]{0.32\textwidth}
		\begin{center}
			\includegraphics[width=1.22in,height=0.8in]{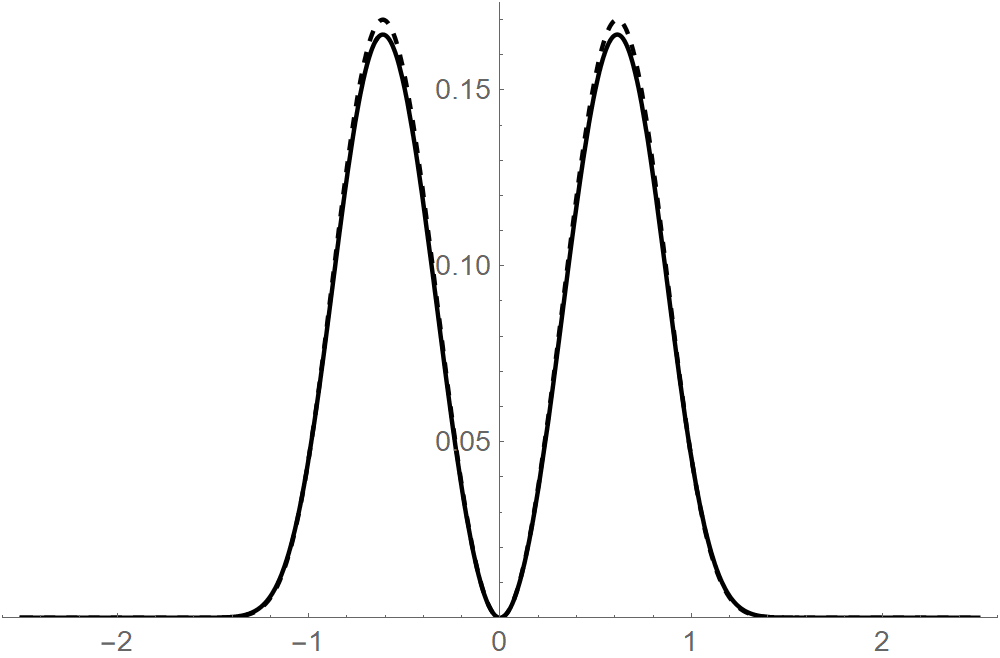}
		\end{center} \subcaption{$N=1000$}
	\end{subfigure}
	
	\caption{ We show cuts of the graphs of the rescaled densities $R_{N,1}(z)$ (full line) and its large-$N$ limit $R_1(z)$ (dashed line) at $z=x+iy$ restricted to $y=0$	 
for plots	(A)~--~(C) and to $x=0$ 
for plots (D)~--~(F) with $\alpha_N=1$ and $\tau=\tau_c$.  
} \label{R_V conv}
\end{figure}

\begin{rem}
Notice that the right/left-most edge points of the droplet \eqref{S alpha 2} are given by 
\begin{equation*}
\pm \sqrt{ (1+\tau\sqrt{1+\alpha})( \sqrt{1+\alpha}+\tau )  }.
\end{equation*}
Furthermore, beyond the critical regime when $\tau > \tau_c$, there are two more edge points on the real axis
\begin{equation*}
\pm \sqrt{(\tau\sqrt{1+\alpha}-1)(\sqrt{1+\alpha}-\tau)}.
\end{equation*}
We expect that in the limit of weak non-Hermiticity when $1-\tau\sim N^{-1/3} \to 0$, the local statistics at such points lies in the same universality class studied in \cite{akemann2010interpolation,MR2594353}, which interpolates the Airy and (boundary) Ginibre point processes. 
\end{rem}

\section{Preliminaries}\label{Prelim}

For given $A > B \ge 0$, we consider the $N$-dependent potentials $Q_N$, $V_N$ 
\begin{align}
Q_{N}(\zeta)&:=\frac{1}{N} \log \frac{1}{  K_\nu(AN|\zeta|) |\zeta|^\nu  }-B\re \zeta;
\\
\label{V_N}
V_N(\zeta)&:=\frac{1}{N} \log \frac{1}{  K_\nu(AN|\zeta|^2) |\zeta|^{2\nu+2}  }-B\re \zeta ^2,
\end{align}
where $K_\nu$ is the modified Bessel function of the second kind given by
\begin{equation*}
K_\nu(z):=\frac{\pi}{2} \frac{I_{-\nu}(z)-I_{\nu}(z)}{\sin \nu \pi}, \quad I_\nu(z):=\sum_{k=0}^{\infty} \frac{(z/2)^{2k+\nu}}{k! \,\Gamma(k+\nu+1)}.
\end{equation*}
Notice here that the potential $V_N$ is related to $Q_N$ as 
\begin{equation} \label{V Q relation}
V_N(\zeta)=Q_N(\zeta^2)-\frac{2}{N}\log|\zeta|.
\end{equation}

When $\tau \not= 1,$ the joint probability density 
$\widehat{\textbf{P}}_N$ of the eigenvalues $\widehat{\boldsymbol{\zeta}}=(\widehat{\zeta}_1,\cdots,\widehat{\zeta}_N)$ of the non-Hermitian Wishart matrix $X$ is of the form 
\begin{equation} \label{Gibbs}
d\textbf{P}_N( \widehat{\boldsymbol{\zeta}} )= \frac{1}{\widehat{Z}_N} e^{- \widehat{\bfH}_N( \widehat{\boldsymbol{\zeta}} ) } \prod_{j=1}^{N} dA( \widehat{\zeta}_j),
\end{equation}
where $\widehat{Z}_N$ is the normalisation constant (also known as the partition function), and $\widehat{\bfH}_N$ is the Hamiltonian given by 
\begin{equation} \label{H d=1}
\widehat{\bfH}_N( \widehat{\boldsymbol{\zeta}} ):=\sum_{j\neq k} \log \frac{1}{| \widehat{\zeta}_j-\widehat{\zeta}_k|}+N\sum_{j=1}^{N} Q_N(\widehat{\zeta}_j).
\end{equation}
Here $A,B$ are related to the non-Hermiticity parameter $\tau$ as
\begin{equation} \label{A B tau}
A=\frac{2}{1-\tau^2}, \quad B=\frac{2\tau}{1-\tau^2}.
\end{equation}
Similarly, the joint probability density $\textbf{P}_N$ of the (random) eigenvalue system $\boldsymbol{\zeta}=(\zeta_1,\cdots,\zeta_N)$ of the Dirac matrix $\mathcal{D}$ is given by 
\begin{equation} \label{Gibbs square}
d\textbf{P}_N(\boldsymbol{\zeta})= \frac{1}{Z_N} e^{- \bfH_N( \boldsymbol{\zeta} ) } \prod_{j=1}^{N} dA(\zeta_j),
\end{equation}
where
\begin{equation} \label{H d=2}
\bfH_N( \boldsymbol{\zeta} ):=\sum_{j\neq k} \log \frac{1}{|\zeta_j^2-\zeta_k^2|}+N\sum_{j=1}^{N} V_N(\zeta_j).
\end{equation}

We emphasise here that when we consider the models as a particle system governed by the Boltzmann–Gibbs law \eqref{Gibbs} or \eqref{Gibbs square}, we allow $\nu$ to be an arbitrary real number as long as $\nu>-1$, and $A, B$ be arbitrary real parameters satisfying $A>B \ge 0.$ When we consider such general set-up, we sometimes use the notation $\tau:=B/A.$

It is well known that the above particle systems have the structure of a determinant, namely, the  $k$-point correlation functions 
\begin{align}
\begin{split}
\widehat{\bfR}_{N,k}(\widehat{\zeta}_1,\cdots,\widehat{\zeta}_k)&:=\frac{1}{\widehat{Z}_N} \frac{N!}{(N-k)!} \int_{ \C^{N-k} }  e^{- \widehat{\bfH}_N( \widehat{\boldsymbol{\zeta}} ) } \prod_{j=k+1}^N  dA( \widehat{\zeta}_j);
\\
\bfR_{N,k}(\zeta_1,\cdots,\zeta_k)&:=\frac{1}{Z_N} \frac{N!}{(N-k)!} \int_{ \C^{N-k} }  e^{- \bfH_N( \boldsymbol{\zeta} ) } \prod_{j=k+1}^N  dA(\zeta_j) \label{ bfR_{N,k} },
\end{split}
\end{align}
are expressed in terms of certain correlation kernels $\widehat{\bfK}_N,\bfK_N$ as 
\begin{align}
\widehat{\bfR}_{N,k}(\widehat{\zeta}_1,\cdots,\widehat{\zeta}_k)&=\det \Big[ \widehat{\bfK}_N( \widehat{\zeta}_j,\widehat{\zeta}_l) \Big]_{ 1 \le j,l \le k } ;
\\
\bfR_{N,k}(\zeta_1,\cdots,\zeta_k)&=\det \Big[ \bfK_N(\zeta_j,\zeta_l) \Big]_{ 1 \le j,l \le k } .
\end{align}

Moreover, it was discovered by Osborn in \cite{osborn2004universal} that the model \eqref{Gibbs square} can be exactly solved due to the orthogonality relation
\begin{equation} \label{Orthogonal relation}
\int_{ \C } L_j^\nu(c \, \zeta^2) \overline{L_k^\nu(c \, \zeta^2 )} e^{-NV_N(\zeta)} \, dA(\zeta)=\frac{h_j^\nu}{2} \, \delta_{jk}, \quad c=\frac{A^2-B^2}{2B}N,
\end{equation}
where $L_j^\nu$ is the generalised Laguerre polynomial (\cite[Chapter \RN{5}]{MR0372517}) given by
\begin{equation} \label{Laguerre poly}
L_j^\nu(z):=\sum_{k=0}^{j} \frac{ \Gamma(j+\nu+1) }{  (j-k)! \, \Gamma(\nu+k+1) } \frac{(-z)^k}{k!}.
\end{equation}
Here, the orthogonal norm $h_j^\nu$ is given by
\begin{equation}
h_j^\nu=\frac{1}{A N^{\nu+2} } \Big( \frac{2A}{A^2-B^2}  \Big)^{\nu+1} \frac{\Gamma(j+\nu+1)}{j!} \Big( \frac{A}{B} \Big)^{2j}.
\end{equation}
The orthogonality relation \eqref{Orthogonal relation} was conjectured for $\nu=\pm 1/2$ in \cite{akemann2002microscopicA} and for general $\nu\in\mathbb{N}$ in \cite{osborn2004universal}.
It was first proved in \cite[Appendix A]{MR2180006}. We also refer the reader to \cite[Proposition 1]{akemann2010interpolation} for an elementary proof 
by induction, based on the contour integral representation of $L_j^\nu$. 
By virtue of the orthogonality relation \eqref{Orthogonal relation} and Dyson's determinantal formula (see e.g., \cite[Chapter 5]{forrester2010log}), the correlation kernel $\bfK_N$ has an expression
\begin{equation} \label{bfK V}
\bfK_N(\zeta,\eta)= e^{-N V_N(\zeta)/2-N V_N(\eta)/2} \sum_{j=0}^{N-1} \frac{ L_j^\nu(c \, \zeta^2) \overline{L_j^\nu (c \, \eta^2)}   }{h_j^\nu/2}.
\end{equation}
In particular, we have
\begin{align} 
\bfR_{N,1}(\zeta)&= 2 A N^{\nu+2} \Big( \frac{A^2-B^2}{2A}  \Big)^{\nu+1} K_\nu(AN|\zeta|^2) |\zeta|^{2\nu+2} 
\\
& \times e^{B N \re \zeta^2} \sum_{j=0}^{N-1} \frac{j!}{\Gamma(j+\nu+1)} \Big( \frac{B}{A} \Big)^{2j} \Big|  L_j^\nu\Big( \frac{A^2-B^2}{2B}N \zeta^2 \Big)  \Big|^2.
\end{align}

By definition, the $k$-point correlation function $R_{N,k}$ of the rescaled point process $\boldsymbol{z}$ is given by 
\begin{equation} \label{ R_{N,k} }
R_{N,k}(z_1,\cdots,z_k)=\frac{1}{(N\delta)^{k/2 }  } \bfR_{N,k}(\zeta_1,\cdots,\zeta_k)=\det \Big[ K_N(z_j, z_l) \Big]_{ 1 \le j,l \le k },
\end{equation}
where 
\begin{equation} \label{rescaling}
K_N(z,w):=\frac{1}{\sqrt{N \delta}} \bfK_N \Big(  \frac{z}{(N \delta)^{1/4} },  \frac{w}{(N \delta)^{1/4} } \Big).
\end{equation}
Thus the rescaled density $R_{N,1}$ has the following expression 
\begin{align} \label{R_{N,1}}
\begin{split}
R_{N,1}(z)
&= 4 \, \nu^{\frac{\nu}{2}+1}  ( 1-\tau^2  )^{\nu+1} K_\nu( 2\sqrt{\nu} |z|^2) |z|^{2\nu+2}
\\
&\times  e^{ 2 \tau \sqrt{\nu} \re z^2} \sum_{j=0}^{N-1} \frac{ \tau^{2j} j! }{\Gamma(j+\nu+1)}  \Big|  L_j^\nu\Big( \frac{1-\tau^2}{\tau}\sqrt{ \nu } z^2 \Big)  \Big|^2.
\end{split}
\end{align}

\section{Global statistics and multi-critical point} \label{Section_global}

\subsection{The asymptotic behaviours of the potentials}
In this subsection, we compile some asymptotic behaviours of the potentials $V_N, Q_N$ when $N \to \infty$. In order to describe the asymptotic behaviours of functions, we often use the following notations, due to Bachmann and Landau. 
If $f(x)/g(x)$ tends to $1$ as $x\to\infty,$ we say that $f$ is asymptotic to $g$ and write $f(x) \sim g(x).$ 

First, we consider the case that the parameter $\nu$ is fixed. 
In this case, since
\begin{equation}\label{K nu asy}
K_\nu(z) \sim \sqrt{ \frac{\pi}{2z}} e^{-z},\quad \text{as } z\to \infty,
\end{equation}
for any fixed $\nu$ (see \cite[Eq.(10.40.2)]{olver2010nist}), the dominant term of the potential $Q_N$ is given by
\begin{equation}\label{Q fixed nu}
\widetilde{Q}_0(\zeta)=A|\zeta|-B \re\zeta.
\end{equation} 
Here and in the sequel, we sometimes use the convention that two potentials are  considered to be equal if they differ by an $N$-dependent constant.
Secondly, assume that the parameter $\nu$ scales with $N$ as \eqref{nu alpha}.
Note that by \cite[Eq.(10.41.4)]{olver2010nist}, as $\nu \to \infty$, 
\begin{equation}\label{K nu asy inf}
K_\nu(\nu |z|) \sim \Big( \frac{\pi}{2\nu} \Big)^{\frac12 } ( 1+|z|^2)^{-\frac14} \Big( \frac{1+\sqrt{1+|z|^2}}{|z|}  \Big)^\nu e^{-\nu \sqrt{1+|z|^2}}
\end{equation}
uniformly for $0<|z|<\infty.$ Therefore for the scale \eqref{nu alpha}, we have
\begin{equation} \label{Q_N asym}
Q_N(\zeta) = \widetilde{Q}_{\alpha}(\zeta)+\frac{1}{4N} \log ( A^2|\zeta|^2+\alpha^2 )+\frac{1}{2N}\log N+o\Big( \frac{1}{N}\Big),
\end{equation}
where 
\begin{equation} \label{Q alpha}
\widetilde{Q}_\alpha(\zeta):=\sqrt{A^2|\zeta|^2+\alpha^2 }-B \re \zeta-\alpha \log \, ( \sqrt{A^2|\zeta|^2+\alpha^2}+\alpha  ).
\end{equation}
Thus on the macroscopic level, one can formally think of the case with fixed $\nu$ as a special case of \eqref{nu alpha} with $\alpha=0.$ This is consistent with the definition in \eqref{Q fixed nu}. 

On the other hand, by the relation \eqref{V Q relation}, the potential $V_N$ tends to 
\begin{equation} \label{V_alpha}
\widetilde{V}_\alpha(\zeta):=\sqrt{A^2|\zeta|^4+\alpha^2 }-B \re \zeta^2-\alpha \log  \, ( \sqrt{A^2|\zeta|^4+\alpha^2}+\alpha  )-\frac{2}{N}\log|\zeta|.
\end{equation}
Note also that the Laplacian $\Delta:=\partial \bar{\partial}$ of the potentials are given by
\begin{equation} \label{density}
\Delta \widetilde{Q}_{\alpha}(\zeta)=\frac{A^2}{4} \frac{1}{\sqrt{ A^2 |\zeta|^2+\alpha^2 }},
\end{equation}
and
\begin{equation}
\Delta \widetilde{V}_\alpha(\zeta)=4|\zeta|^2 \Delta \widetilde{Q}_{\alpha}(\zeta^2)= \frac{A^2 |\zeta|^2}{\sqrt{ A^2 |\zeta|^4+\alpha^2 }}.
\end{equation}
In particular, we have 
$\Delta \widetilde{V}_0(\zeta)=A.$

\subsection{A non-Hermitian generalisation of the Marchenko-Pastur law} 
Let us denote by $\widehat{\E}_N$ the expectation with respect to the measure $\widehat{\textbf{P}}_N.$ 
Then due to Johansson's marginal measure theorem for the plane (see \cite[Theorem  2.9]{HM13} and \cite[Theorem 2.1]{MR1487983}), for each bounded continuous function $f$, we have 
\begin{equation}
\frac{1}{N} \widehat{\E}_N\Big[ f(\widehat{\zeta}_1)+\cdots +f(\widehat{\zeta}_N) \Big] \to \int f \, d \widehat{\mu}, 
\end{equation}
where $\widehat{\mu}$ is Frostman's equilibrium measure (i.e., a unique  minimiser among compactly supported probability measures) of the weighted logarithmic energy functional
\begin{equation} \label{I}
	I[\sigma]:=\int_{\C^2 } \log \frac{1}{|\zeta-\eta|^2} \, d\sigma(\zeta)\, d\sigma(\eta) +2\int_{\C} \widetilde{Q}_{\alpha}(\zeta) \,d\sigma(\zeta).
\end{equation}
We emphasise here that one has to slightly generalise the marginal measure theorem by the $N$-dependence of the potential. 
To be more precise, by adapting fairly standard arguments from \cite{HM13,MR1487983}, the same conclusion can be drawn for an $N$-dependent potential $Q_N$ of the form $Q_N=Q+u/N$, where $u$ is a continuous function. 
Thus the asymptotic behaviour \eqref{Q_N asym} allows us to apply this result. 

It is well known that the equilibrium measure $\widehat{\mu}$ is characterised by the variational conditions (\cite[p.27]{ST97}):
\begin{align} \label{E-L 1}
\begin{split}
\int \log \frac{1}{|\zeta-z|}\, d\widehat{\mu}(z)+\frac{\widetilde{Q}_\alpha(\zeta)}{2}&=c, \quad \text{q.e.,\quad if }\zeta \in \widehat{S}_\alpha;
\\
\int \log \frac{1}{|\zeta-z|}\, d\widehat{\mu}(z)+\frac{\widetilde{Q}_\alpha(\zeta)}{2}& \ge c, \quad \text{q.e.,\quad otherwise.}
\end{split}
\end{align} 
Here, $c$ is called the modified Robin constant. 
Moreover, $\widehat{\mu}$ is absolutely continuous with respect to the Lebesgue measure and takes the form 
\begin{equation} \label{hat mu density}
d\widehat{\mu}=\Delta \widetilde{Q}_{\alpha} \cdot \mathbbm{1}_{ \widehat{S}_\alpha } \, dA
\end{equation}
for a certain compact set $\widehat{S}_\alpha$ called the \emph{droplet}. We refer to \cite{ST97} for a standard reference on logarithmic potential theory. Thus the density \eqref{NWishart_density} in Theorem~\ref{Thm_NWishart} immediately follows from \eqref{A B tau} and \eqref{density}.  Therefore, to describe the macroscopic behaviour of the particle system \eqref{Gibbs} in the large-$N$ limit, it suffices to characterise the shape of the droplet $\widehat{S}_\alpha$. 

\medskip 
Before dealing with the general set-up, let us recall the well-known extremal cases. 

	\begin{itemize}
		\item \textit{The maximally non-Hermitian limit:} $\tau=0$. 
		Recall that for a radially symmetric potential $q(|\zeta|)$, the droplet is given by an annulus $\{ z \in \C :  R_1 \le |z| \le R_2 \}$, where $R_1, R_2$ are the unique pair of constants satisfying
		$$
		R_1 q'(R_1)=0, \quad R_2 q'(R_2)=2, 
		$$
		see \cite[Section \RN{4}.6]{ST97}. 
		Note that for $\tau=0$, we have
		$$
	\widetilde{Q}_{\alpha}(\zeta)=q_\alpha(|\zeta|), \quad q_\alpha(r):=\sqrt{4r^2+\alpha^2}-\alpha \log  \, ( \sqrt{4 r^2+\alpha^2}+\alpha )
		$$
		and
		$$
		r q_\alpha'(r)=\frac{4r^2}{ \sqrt{4r^2+\alpha^2}+\alpha }. 
		$$
		Therefore the associated droplet is the disc with centre the origin and radius  $\sqrt{1+\alpha}$.
		\item \textit{The Hermitian limit:} $\tau \uparrow 1$. In this case, one can observe from the expression \eqref{Q alpha} that 
		$$
		\lim_{\tau \uparrow 1} \widetilde{Q}_{\alpha}(x+iy)= \begin{cases}
		x-\alpha \log x &\text{for } x>0, \hspace{0.5em} y=0,
		\\
		\quad \quad \infty &\text{otherwise.}
		\end{cases} 
		$$
		Thus in the Hermitian limit, by solving the associated equilibrium problem on the positive real axis,	one can observe that the macroscopic density follows the Marchenko-Pastur law \eqref{MP}, see e.g., \cite[Section 3.4]{forrester2010log}.
	\end{itemize}

For general $A>B \ge 0$ and $\alpha \ge 0$, we will obtain that the droplet $\widehat{S}_\alpha$ is enclosed by an ellipse 
\begin{equation} \label{S alpha}
\frac{(x-x_0)^2}
{ a ^2 } + \frac{y^2 }{ b^2 }   = 1,
\end{equation}
where $\tau=B/A$ and
\begin{equation}
x_0:=\dfrac{2}{A}\frac{\tau}{1-\tau^2}(2+\alpha), \quad a:=\dfrac{2}{A} \dfrac{1+\tau^2}{1-\tau^2} \sqrt{1+\alpha}, \quad b:= \dfrac{2}{A} \sqrt{1+\alpha}.
\end{equation}
This recovers Theorem~\ref{Thm_NWishart} when we set the parameters $A,B$ as \eqref{A B tau}.

Under the assumption that $\widehat{S}_\alpha^c$ is a simply connected domain, we derive \eqref{S alpha} by means of the conformal mapping method, which is widely used in the theory of  Hele-Shaw flow (see e.g., \cite{MR2245542}). Then we show the simply connectedness of $\widehat{S}_\alpha^c$ using the variational conditions \eqref{E-L 1}. On the other hand, for the general theory on the connectivity of the droplet associated with Hele-Shaw type potentials, we refer the reader to \cite{MR3454377} and references therein. 

To prove Theorem~\ref{Thm_NWishart}, we use the Schwarz function associated with $\widehat{S}_\alpha$ and the Cauchy transform of the equilibrium measure $\widehat{\mu}.$
For a given domain $\Omega \subsetneq \C$ such that $\infty \notin \partial \Omega$, the \textit{Schwarz function} $F: \bar{\Omega} \to \C\cup \{ \infty\}$ is a unique meromorphic function (if it exists) satisfying 
\begin{equation}
F(\zeta)=\bar{\zeta} \quad \text{on } \partial \Omega.
\end{equation}
The Cauchy transform $C_\mu$ of the measure $\mu$ is defined by
\begin{equation}
C_\mu(\zeta):=\lim_{\eps \to 0}  \int_{|\zeta-z|>\eps} \frac{d\mu(z)}{\zeta-z},
\end{equation}
if the limit exists. 

We now prove Theorem~\ref{Thm_NWishart}.

\begin{proof}[Proof of Theorem~\ref{Thm_NWishart}]
	Let $\widehat{\mu}$ be the equilibrium measure associated with the potential $\widetilde{Q}_{\alpha}$ and write $ \widehat{S} \equiv \widehat{S}_\alpha$ for the support of $\widehat{\mu}$. As explained above, it suffices to characterise the shape of the droplet $\widehat{S}$. 
	
	We first find a candidate $\widehat{S}$, assuming that the complement of $\widehat{S}$ is a simply connected domain.
	For this, let $f$ be the conformal map $(\bar{\mathbb{D}}^c,\infty) \to (\widehat{S}^c,\infty)$ such that 
	\begin{equation}
	\label{f-def}
	f(\zeta)=R\, \zeta +q+O\Big( \frac{1}{\zeta} \Big), \quad \text{as }\zeta \to \infty, 
	\end{equation}
	where $R$ is positive, called the conformal radius of $\widehat{S}$.
	
	\medskip
    We prove the theorem by using the following steps: 
    
    \begin{enumerate}[label=\arabic*.] 
     \smallskip	\item we define the analytic continuation of $f$ to $\C \backslash \{0\}$, using the Schwarz function associated with $\widehat{S}$;
     \smallskip	\item we show that the conformal map $f$ is given by the (translated) \textit{Joukowsky transform}
    	\begin{equation*}
    	f(z)=\frac{2}{A}\frac{\sqrt{1+\alpha}}{1-\tau^2}\Big(  z+ \frac{\tau^2}{z} \Big)+x_0,
    	\end{equation*}
    	which maps the unit circle $\partial \mathbb{D}$ to the ellipse $\partial \widehat{S}$;
     \smallskip	\item we show that this candidate $\widehat{S}$ is indeed the desired droplet using the variational conditions.
    \end{enumerate}

   We begin with proving step 1.
   
   \medskip 
  \noindent \textit{Step 1.} Let us first express the Schwarz function $F$ associated with the droplet $\widehat{S}$ in terms of the Cauchy transform $C \equiv C_{\widehat{\mu}}$ of the equilibrium measure $\widehat{\mu}$. 
  By differentiating \eqref{E-L 1}, we have
    \begin{equation} \label{E-L differentiate}
    \partial_\zeta \widetilde{Q}_\alpha(\zeta)=C(\zeta),
    \end{equation}	
    in $\widehat{S}.$ If one can ``separate'' the variable $\bar{\zeta}$ in the above identity, the desired expression of $F$ is derived. Indeed, this procedure is not available for general potentials, but it turns out that the potential $\widetilde{Q}_{\alpha}$ under consideration is one of the very special cases in which this can be done. 
	
	More precisely, by \eqref{Q alpha}, we have 
	\begin{equation} \label{partial Q alpha} 
		\partial_\zeta \widetilde{Q}_{\alpha}(\zeta)
	=\frac{1}{2\zeta} (  \sqrt{A^2|\zeta|^2+\alpha^2}-\alpha  )-\frac{B}{2} .
	\end{equation}
    By \eqref{E-L differentiate}, we observe
	$$
	\bar{\zeta}=\frac{\zeta}{A^2} \Big[ \Big( 2 C(\zeta)+B+\frac{\alpha}{\zeta} \Big)^2- \Big(  \frac{\alpha}{\zeta} \Big)^2  \Big]
	$$
	 in $\widehat{S}.$ 
	Therefore the Schwarz function $F$ associated with $\widehat{S}$ exists and it is expressed in terms of $C$ as 
	\begin{equation} \label{Schwarz}
		F(\zeta)=\frac{\zeta}{A^2}\Big( 2C(\zeta)+B \Big)\Big( 2C(\zeta)+B+\frac{2\alpha}{\zeta} \Big).
	\end{equation}
    Note that by \eqref{Schwarz}, for $z \in \partial \mathbb{D},$
    \begin{equation}
    \overline{f( 1/\bar{z} )}=\overline{f(z)}
    = \frac{f(z)}{A^2} \Big( 2C(f(z))+B \Big)\Big( 2C(f(z))+B+\frac{2\alpha}{f(z)} \Big) .
    \end{equation}
    Now let us define $f: \overbar{\mathbb{D}}\backslash \{0\} \to \C$ by analytic continuation as 
    \begin{equation} \label{f analy cont}
    f(z):=\overline{ \frac{f( 1/\bar{z}  )}{A^2} \Big( 2C(f(1/\bar{z}))+B \Big)\Big( 2C(f(1/\bar{z}))+B+\frac{2\alpha}{f(1/\bar{z})} \Big)   }. 
    \end{equation}
    Therefore we have obtained the analytic continuation $f$ of \eqref{f-def} to $\C \backslash\{0\}$. Here and henceforth, we abuse notation by letting $f$ denote the analytic continuation of $f$ to $\C \backslash\{0\}$.
    
    \medskip
    \noindent \textit{Step 2.} By construction, complex infinity is the only (simple) pole of $f$ in $\mathbb{D}^c$. On the other hand, notice here that since $\widehat{\mu}(\C)=1$, we have
    \begin{equation}
     \label{C near inf}
    C(\zeta)=\frac{1}{\zeta}+O\Big( \frac{1}{\zeta^2} \Big), \quad \text{as }\zeta \to \infty.
    \end{equation}
      Therefore from \eqref{Schwarz} and \eqref{f analy cont}, one can observe that in $\mathbb{D}$, the origin is the unique (simple) pole of $f$.
     Combining all of the above, we obtain that the function $f:\widehat{\C} \to \widehat{\C}$ is analytic except for the simple poles $0$ and $\infty$. In other words, we have shown that the map $f$ is of the form 
     \begin{equation} \label{f rational}
     f(z)=R \, z+q+\frac{R^*}{z}. 
     \end{equation}

    Now we determine three constants $R, q$ and $R^*$ in terms of $\alpha$ and $\tau$. To obtain interrelations among the constants, we calculate the asymptotic behaviour of $f$ near the origin using \eqref{f analy cont}.
    For this, note that 
    \begin{equation}
      \overline{f(1/\bar{z})}=\frac{R}{z}+q+O(z), \quad 	\overline{ C( f(1/\bar{z}) ) }=\frac{z}{R}+O(z^2), \quad \text{as }z \to 0.
    \end{equation}
    By \eqref{f analy cont}, $f$ has the Laurent series expansion about the origin
    \begin{equation} \label{f rational 2}
    f(z)   = \frac{\tau^2 R}{z}+ \tau^2 q+ \frac{2\tau(2+\alpha)}{A}+O(z).
    \end{equation}
    Comparing the coefficient of $1/z$ of \eqref{f rational} and that of \eqref{f rational 2}, we find
    \begin{equation}
    R^*=\tau^2 R.
    \end{equation}
    Moreover, by comparing the constant terms, we have
    \begin{equation}
    \tau^2 q+ \frac{2\tau(2+\alpha)}{A}= q,
    \end{equation}
    which leads to
    \begin{equation}
    q=\dfrac{2}{A}\frac{\tau}{1-\tau^2}(2+\alpha)=x_0.
    \end{equation}
    
    Combining all of the above, we have shown that $f$ is a rational function of the form 
    \begin{equation} \label{f} 
    f(z)=R \, \Big( z+ \frac{\tau^2}{z} \Big)+x_0.
    \end{equation} 
    In other words, $\widehat{S}$ is of the form
     \begin{equation} \label{S hat R}
    \widehat{S}:= \Big\{ \zeta=x+iy:  \Big(  \frac{x-x_0}{R\,(1+\tau^2)}  \Big)^2+ \Big(  \frac{y}{R\,(1-\tau^2)}  \Big)^2 \le 1 \Big\}.
    \end{equation}
    
     Next, we show
    \begin{equation} \label{R}
    R=\frac{2}{A}\frac{\sqrt{1+\alpha}}{1-\tau^2}.
    \end{equation}
    Since $\widehat{\mu}$ is the equilibrium measure of mass-one, it follows from \eqref{hat mu density} that 
    \begin{equation}\label{mass 1 gen}
   \frac{A^2}{4}  \int_{\widehat{S}} \frac{1}{\sqrt{ A^2 |z|^2+\alpha^2 }} \, dA(z)=1.
    \end{equation}
    Notice here that since the semi-major/minor axis of $\widehat{S}$ is linear in $R$, the left-hand side of \eqref{mass 1 gen} is an increasing function of $R$. Therefore it suffices to show \eqref{mass 1 gen} with the choice of \eqref{R}. 
   
   We first consider the case $0 \notin \partial \widehat{S}$. The special case $0 \in \partial \widehat{S}$ will be treated later as a limiting case. Let us write $\mathbb{D}_\eps$ for the disc with centre the origin and  a sufficiently small radius $\eps$.
   Then, by applying Green's formula to the domain $\widehat{S} \backslash \mathbb{D}_\eps$, the left-hand side of \eqref{mass 1 gen} is computed as
   \begin{equation}
   \label{rhs Green}
   \frac{1}{2\pi i} \int_{\partial \widehat{S} } \frac{\sqrt{A^2|z|^2+\alpha^2 }}{2z} \, dz -\frac{\alpha}{2} \cdot \mathbbm{1}_{ \{ 0 \in \textrm{int}(\widehat{S}) \} }.
   \end{equation}
   In the special case $0 \in \partial \widehat{S}$, the line integral should be understood as a principal value. 
  Using the change of variable $z=f(w)$, we have 
  \begin{equation}
    \int_{\partial \widehat{S} } \frac{\sqrt{A^2|z|^2+\alpha^2 }}{2z} \, dz
  =\int_{\partial \mathbb{D}} \frac{ \sqrt{ A^2  f(w) \overline{ f(1/\bar{w})}   +\alpha^2 }}{2f(w)} f'(w) \, dw
  =\int_{\partial \mathbb{D}} g(w)  \, dw,
  \end{equation}
    where $g$ is a rational function of the form
    \begin{equation} \label{g}
     g(w) :=\frac{\tau}{1-\tau^2}\sqrt{1+\alpha} \frac{   w^2+\frac{1+\tau^2}{2\tau}\frac{2+\alpha}{\sqrt{1+\alpha}}w+1   }{ w^2 + \tau \frac{2+\alpha}{\sqrt{1+\alpha}} w+ \tau^2     } \Big(  1- \frac{\tau^2}{w^2}  \Big).
    \end{equation}
   Here, we use the fact that 
    \begin{align}
    \begin{split}
    A^2  f(w) \overline{ f(1/\bar{w})}   +\alpha^2&=\Big( B R \, w+\frac{1+\tau^2}{1-\tau^2}(2+\alpha)+\frac{B R}{w} \Big)^2
    \\
    & +A^2 (1-\tau^2)^2 \Big( R^2-\frac{4}{A^2} \frac{1+\alpha}{(1-\tau^2)^2}  \Big)
    \end{split}
    \end{align}
    becomes the square of a rational function, if we choose $R$ as presented in \eqref{R}.
    
    Note that $g$ has poles only at 
    \begin{equation} \label{p1 p2}
    0, \quad p_1:=-\tau/\sqrt{1+\alpha},\quad p_2:=-\tau\sqrt{1+\alpha},
    \end{equation}
    with the residues
    \begin{equation}
      1+\frac{\alpha}{2}, \quad -\frac{\alpha}{2}, \quad \frac{\alpha}{2},
    \end{equation}
    respectively.
    
    Notice here that since $A>B \ge 0$, the pole $p_1$ is always contained in $\mathbb{D}$. 
    On the other hand, due to the formula \eqref{S hat R},
    $$
  \partial \widehat{S} \cap \R = \Big\{  \frac{2}{B} \frac{\tau^2}{1-\tau^2}(2+\alpha)\pm R \, (1+\tau^2)  \Big\} .
    $$
    Therefore with the choice of \eqref{R}, we have 
    $$
    0\in \textrm{int}(\widehat{S}) \quad \textrm{if and only if} \quad
    \tau < 1/\sqrt{1+\alpha}.
    $$
    Thus we observe that 
    $$ p_2 \in \mathbb{D} \quad \textrm{if and only if} \quad 0\in \textrm{int}(\widehat{S}).$$ 
    By the residue calculus we obtain that 
    \begin{equation}\label{int g}
    \frac{1}{2\pi i} \int_{\partial \mathbb{D}} g(w) \, dw=1+\frac{\alpha}{2} \cdot \mathbbm{1}_{ \{ 0 \in \textrm{int}(\widehat{S}) \} }.
    \end{equation}
    Combining \eqref{rhs Green} and \eqref{int g}, the desired equation \eqref{mass 1 gen} follows in the case $0 \notin \partial \widehat{S}$. 

    We now turn to the special case $0 \in \partial \widehat{S}$. In this case, we have 
    \begin{equation}
    \tau=\tau_c:=1/\sqrt{1+\alpha}, \quad p_2:=-\tau \sqrt{1+\alpha}=-1 \in \partial \mathbb{D}, \quad 0=f(-1).
    \end{equation}
    Let us denote by $\mathbb{D}_r(\eta)$ the disc with centre $\eta \in \C$ and radius $r.$ For a sufficiently small $\eps>0$, we consider 
    \begin{equation}
    \widehat{S}_\eps:=\widehat{S} \, \backslash \, f (  \mathbb{D}_\eps(-1) ).
    \end{equation}
    Then the boundary of $f^{-1}(\widehat{S}_\eps)=\mathbb{D}  \backslash  \mathbb{D}_\eps(-1)$ is the union of two arcs of circles: 
    \begin{equation}
    \partial f^{-1}(\widehat{S}_\eps)= \Upsilon_\eps^1 \cup \Upsilon_\eps ^2, \quad  \Upsilon_\eps^1  \subseteq \partial \mathbb{D}, \quad  \Upsilon_\eps^2  \subseteq \partial \mathbb{D}_\eps(-1),
    \end{equation}
    where $\Upsilon_\eps^1$ is oriented in the anticlockwise direction and $\Upsilon_\eps^2$ in the clockwise direction. Applying Green's formula to $\widehat{S}_\eps$, we obtain    
    \begin{equation} \label{Green spec}
    \frac{A^2}{4}  \int_{\widehat{S}_\eps} \frac{1}{\sqrt{ A^2 |z|^2+\alpha^2 }} \, dA(z)=\frac{1}{2\pi i}  \int_{ \partial \widehat{S}_\eps}  \frac{\sqrt{A^2|z|^2+\alpha^2 }}{2z} \, dz.
    \end{equation}
    Using the change of variable $z=f(w)$ again, the above integral is computed as
    \begin{equation}
    \frac{1}{2\pi i} \int_{ \Upsilon_\eps^1 \cup \Upsilon_\eps^2  } g(w) \, dw+\frac{1}{2\pi i} \int_{ \Upsilon_\eps^2 } \widetilde{g}(w) \,dw, 
    \end{equation}
    where 
    \begin{equation}
    \widetilde{g}(w):=\frac{\sqrt{A^2|f(w)|^2+\alpha^2}}{2f(w)} f'(w)-g(w).
    \end{equation}
    Note that $\widetilde{g}$ vanishes on the unit circle. 
    By the residue calculus we have 
    \begin{equation}
    \frac{1}{2\pi i} \int_{ \Upsilon_\eps^1 \cup \Upsilon_\eps^2  } g(w) \, dw=   \underset{ \strut w=0 }{\textrm{Res}}  \hspace{0.35em} g(w)+ \hspace{-0.5em} \underset{ \strut \hspace{0.5em} w=p_1 }{\textrm{Res}} \hspace{-0.15em} g(w)=1.
    \end{equation}  
    On the other hand, as $\eps \to 0$ 
    \begin{equation}
    \frac{1}{2\pi i} \int_{ \Upsilon_\eps^2 } \widetilde{g}(w) \,dw \to 0,
    \end{equation}
    since $\textrm{length}(\Upsilon_\eps^2) \to 0$ and $\sup\limits_{ w \in \mathbb{D}_\eps(-1)  } |\widetilde{g}(w)| \to 0$ as $\eps \to 0$. Taking the limit in \eqref{Green spec} as $\eps \to 0$, we obtain the desired equation \eqref{mass 1 gen} in the case $0 \in \partial \widehat{S}$.
      
\medskip 

\noindent \textit{Step 3.} Next, we show the variational conditions \eqref{E-L 1}.
Set 
\begin{equation}
H(\zeta):=\int \log \frac{1}{|\zeta-z|} \, d\widehat{\mu}  (z)+\frac{ \widetilde{Q}_{\alpha}(\zeta)    }{2}. 
\end{equation}
By definition, we have
\begin{equation} \label{H diff}
2\partial_\zeta H(\zeta)=\partial_\zeta \widetilde{Q}_{\alpha}(\zeta) - C(\zeta).
\end{equation}
Note that by \eqref{hat mu density}, we have
\begin{equation}
C(\zeta)=  \frac{A^2}{4}  \int_{\widehat{S}} \frac{1}{\zeta-z}\frac{1}{\sqrt{ A^2 |z|^2+\alpha^2 }} \, dA(z).
\end{equation}
We first compute $C(\zeta)$ for $\zeta \not=0$. We only consider the case $ 0,\zeta  \notin \partial \widehat{S}$. The other special case can be treated as a limiting case using a similar argument as in step 2.

Choosing a sufficiently small $\eps>0$ such that $\mathbb{D}_\eps(0) \cap \mathbb{D}_\eps(\zeta )=\emptyset$ and applying Green's formula to the domain $\widehat{S} \backslash ( \mathbb{D}_\eps(0) \cup \mathbb{D}_\eps(\zeta ) ) $, we have 
\begin{align} \label{Cauchy eq 1}
\begin{split}
C(\zeta)&= \frac{1}{2\pi i} \int_{\partial \widehat{S} } \frac{1}{\zeta-z}\frac{\sqrt{A^2|z|^2+\alpha^2 }}{2z} \, dz
\\
& -\frac{\alpha}{2\zeta} \cdot \mathbbm{1}_{ \{ 0 \in \textrm{int}(\widehat{S}) \} } +\frac{\sqrt{A^2|\zeta|^2+\alpha^2 }}{2\zeta} \cdot \mathbbm{1}_{ \{ \zeta \in \textrm{int}(\widehat{S}) \} }.
\end{split}
\end{align}
Again, if either $0$ or $\zeta$ lies on $\partial \widehat{S}$, then the line integral should be understood as a principal value. Using the change of variable $w=f(z)$, we have
\begin{equation}
\int_{\partial \widehat{S} } \frac{1}{\zeta-z}\frac{\sqrt{A^2|z|^2+\alpha^2 }}{2z} \, dz
=  \int_{\partial \mathbb{D} } h_\zeta(w) \, dw, \quad h_\zeta(w):=\frac{g(w)}{\zeta-f(w)},
\end{equation}
where $g$ is the rational function given by \eqref{g}. 

Observe that for each $\zeta \in \C$, the pre-image $f^{-1} \{ \zeta \}$ has two elements $\{ w_\zeta^\pm \}$ (counted with multiplicity). Indeed, $w_\zeta^\pm$ are zeros of the following quadratic equation 
$$
R \, w^2-(\zeta-x_0) w+R\, \tau^2=0.
$$
Therefore the function $h_\zeta$ has poles only at  
\begin{equation*}
0, \quad p_1,\quad p_2, \quad w_\zeta^+, \quad w_\zeta^-, 
\end{equation*}
where $p_1,p_2$ are given by \eqref{p1 p2}. Recall that $p_1 \in \mathbb{D}$, whereas $p_2 \in \mathbb{D}$ if and only if $0\in \textrm{int}(\widehat{S})$. Also notice that the locations of $f^{-1}\{ \zeta \}$ are determined by the location of $\zeta$ in the following way: 
\begin{itemize}
	\item $\zeta \in \textrm{int}(\widehat{S})$ if and only if both of $f^{-1}\{ \zeta \}$ are in $\mathbb{D}$;
	\item  $\zeta \in \widehat{S}^c$ if and only if one of $f^{-1}\{ \zeta \}$ is in $\mathbb{D}$ and the other is in $\overbar{\mathbb{D}}^c$.
\end{itemize}
(In the latter case, we denote by $w_\zeta^-$ (resp., $w_\zeta^+$) the pre-image of $\zeta$ in $\mathbb{D}$ (resp., $\mathbb{D}^c$).)

To see the second claim, we first note that if $\zeta \in \widehat{S}^c$ and we choose $w_\zeta^+=(f|_{\overbar{\mathbb{D}}^c }  )^{-1}(\zeta)$, then obviously $w_\zeta^+ \in \mathbb{D}^c.$ Since $w_\zeta^+ w_\zeta^-=\tau^2<1$, we have $w_\zeta^- \in \mathbb{D}.$

Conversely, if one of $f^{-1}\{\zeta\}$, say $w_\zeta^+$ is in $\overbar{\mathbb{D}}^c$, then $w_\zeta^+=(f|_{\overbar{\mathbb{D}}^c }  )^{-1}(\eta)$ for some $\eta \in \widehat{S}^c.$ Since $f|_{\overbar{\mathbb{D}}^c }$ maps $\overbar{\mathbb{D}}^c$ conformally onto $\widehat{S}^c$, we have $\eta=\zeta\in \widehat{S}^c.$

The first claim follows immediately from the second one; $\zeta \in \textrm{int}(\widehat{S})$ if and only if both of $f^{-1}(\zeta)$ are in $\mathbb{D}$ or both are in $\overbar{\mathbb{D}}^c$. However both cannot be in $\overbar{\mathbb{D}}^c$ since $w_\zeta^+ w_\zeta^-=\tau^2<1$.

The residues of $h_\zeta$ at $0,p_1,p_2$ are given by 
\begin{equation} \label{residue 0 p12}
\frac{A}{2\tau}, \quad -\frac{\alpha}{2\zeta}, \quad \frac{\alpha}{2\zeta}, 
\end{equation}
respectively. On the other hand, to compute the residues of $h_\zeta$ at $w_\zeta^\pm$, we recall that $g$ is the rational function given by \eqref{g} or by 
$$
g(w)=\frac{f'(w)}{2f(w)} \Big( B R \, w+\frac{1+\tau^2}{1-\tau^2}(2+\alpha)+\frac{B R}{w} \Big).
$$
By the residue calculus, we obtain
\begin{equation} \label{residue w pm}
\underset{ \strut \hspace{0.5em} w=w_\zeta^\pm}{\textrm{Res}} h_\zeta(w)=-\frac{ g( w_\zeta^\pm ) }{ f'( w_\zeta^\pm ) }=-\frac{A}{4} \Big( \tau+\frac{1}{\tau} \Big)+\frac{A}{4} \Big( \frac{1}{\tau}-\tau \Big) \frac{1}{\zeta} \Big( 2 R \, w_\zeta^\pm-(\zeta-x_0) \Big).
\end{equation}
Since $w_\zeta^+ + w_\zeta^-=(\zeta-x_0)/R$, we have
$$
\underset{ \strut \hspace{0.5em} w=w_\zeta^+}{\textrm{Res}} h_\zeta(w)+\hspace{-0.5em}\underset{ \strut \hspace{0.5em} w=w_\zeta^-}{\textrm{Res}} h_\zeta(w)=-\frac{A}{2} \Big( \tau+\frac{1}{\tau} \Big)=-\frac{B}{2}-\frac{A}{2\tau}.
$$
Combining all of the above with \eqref{Cauchy eq 1}, we obtain
\begin{equation}
C(\zeta)=\frac{1}{2\zeta} (  \sqrt{A^2|\zeta|^2+\alpha^2}-\alpha   )-\frac{B}{2}, \quad (\zeta \in \widehat{S} ).
\end{equation}
Therefore by \eqref{H diff}, \eqref{partial Q alpha}, and the real-valuedness of $H$, 
we derive the first half of the variational conditions \eqref{E-L 1}. 

Now it remains to show the variational inequality in \eqref{E-L 1}. First notice that since $\widetilde{Q}_\alpha(\zeta) \gg \log|\zeta|$ near infinity, we have 
$$
H(\zeta)\to \infty, \quad \textrm{as } |\zeta| \to \infty.
$$ 

Suppose that the variational inequality in \eqref{E-L 1} fails at some point $\zeta \in \widehat{S}^c$. Then the function $H$ has a critical point in $\widehat{S}^c$, or 
\begin{equation} \label{H critical}
\partial_\zeta \widetilde{Q}_{\alpha}(\zeta)=C(\zeta)
\end{equation}
for some $\zeta \in \widehat{S}^c.$

We consider two holomorphic functions in $\widehat{S}^c$ 
$$
w^-(\zeta):=\frac{\zeta-x_0-\sqrt{(\zeta-x_0)^2-4R^2\tau^2} }{2R}, \quad w^+(\zeta):=\frac{\tau^2}{w^-(\zeta)},
$$
where the branch of the square root is chosen such that 
$$
w^-(\zeta) \to 0, \quad \textrm{as } \zeta \to \infty.
$$
Then by \eqref{Cauchy eq 1}, \eqref{residue 0 p12}, and \eqref{residue w pm}, we have 
\begin{equation} \label{C out S}
C(\zeta)= -\frac{\alpha}{2\zeta}+\frac{A}{4}\frac{1-\tau^2}{\tau}  \frac{\zeta-  \sqrt{(\zeta-x_0)^2-4R^2\tau^2} }{\zeta}, \quad (\zeta \in \widehat{S}^c).
\end{equation}
We remark that with this choice of the branch, the required asymptotic behaviour \eqref{C near inf} holds. 

By \eqref{C out S}, equation \eqref{H critical} is equivalent to
\begin{equation}  \label{H critical 1}
\frac{1+\tau^2}{2\tau}=\frac{1}{\zeta} \sqrt{|\zeta|^2+\alpha^2/A^2}+\frac{1-\tau^2}{2\tau} \frac{1}{\zeta}  \sqrt{(\zeta-x_0)^2-4R^2\tau^2}
\end{equation}
for some $\zeta \in \widehat{S}^c.$ By construction, we have
$$
\sqrt{(f(w^+(\zeta))-x_0)^2-4R^2\tau^2}=R \, \Big(w^+(\zeta)-\frac{\tau^2}{w^+(\zeta)}\Big).
$$
Therefore the equation \eqref{H critical 1} holds if any only if 
$$
R\, \tau\Big( w^+(\zeta)+\frac{1}{w^+(\zeta)} \Big)+\frac{1+\tau^2}{2\tau} x_0=\sqrt{|\zeta|^2+\alpha^2/A^2}. 
$$
Notice that the right-hand side of this equation is real-valued, whereas the left-hand side is real-valued if and only if $w^+(\zeta) \in \partial \mathbb{D}$. 
However, $w^+$ is a function from $\widehat{S}^c$ into $\overbar{\mathbb{D}}^c$. This contradiction shows that the variational inequality in \eqref{E-L 1} holds for each $\zeta \in \widehat{S}^c$. 

Thus $\widehat{S}$ in \eqref{S hat R} is the droplet for the equilibrium measure $\widehat{\mu}$ associated with the potential $\widetilde{Q}_{\alpha}.$
\end{proof} 

In the proof of Theorem~\ref{Thm_NWishart}, we have shown that there is no exceptional point in the variational conditions \eqref{E-L 1}.
\medskip

We now prove Corollary~\ref{Macro}. 

\begin{proof}[Proof of Corollary~\ref{Macro}]
	Due to the orthogonality relation \eqref{Orthogonal relation} and the change of variable $\zeta \mapsto \zeta^2,$ we have 
	\begin{equation} \label{Orthogonal relation 1}
	\int_{ \C } L_j^\nu(c \, \zeta) \overline{ L_k^\nu(c \, \zeta ) } e^{-NQ_N(\zeta)} \, dA(\zeta)=h_j^\nu \, \delta_{jk}.
	\end{equation}
	Thus by Dyson's determinantal formula, the correlation kernel $\widehat{\bfK}_{N}$ of the particle system \eqref{Gibbs} has an expression
	\begin{equation} 
	\widehat{\bfK}_{N}(\zeta,\eta)= e^{-N Q_N(\zeta)/2-N Q_N(\eta)/2} \sum_{j=0}^{N-1} \frac{ L_j^\nu(c \, \zeta) \overline{L_j^\nu(c \, \eta)} }{h_j^\nu}.
	\end{equation}
	In particular by \eqref{bfK V}, we have the relation
	\begin{equation} \label{R_N V,Q rel}
	\bfK_{N}(\zeta,\eta)= 2 |\zeta \eta|\,  \widehat{\bfK}_{N}(\zeta^2,\eta^2).
	\end{equation} 
	Notice here that 
	\begin{equation}
	\zeta \in S_\alpha \quad  \textrm{if and only if}  \quad \zeta^2 \in \widehat{S}_\alpha. 
	\end{equation}
	Therefore by Theorem~\ref{Thm_NWishart}, for $\zeta \in \C$, we have 
	\begin{align}
	\begin{split}
	\lim_{N \to \infty} \frac{\bfR_{N,1}(\zeta)}{N}&= \lim_{N \to \infty} 2|\zeta|^2\frac{ \widehat{\bfR}_{N,1}(\zeta^2) }{N}
	\\
	& =2|\zeta|^2 \Delta \widetilde{Q}_{\alpha}(\zeta^2) \cdot  \mathbbm{1}_{\widehat{S}_\alpha}(\zeta^2) 
	=\frac{\Delta \widetilde{V}_\alpha(\zeta)}{2}\cdot  \mathbbm{1}_{S_\alpha}(\zeta),
	\end{split}
	\end{align}
	which completes the proof.
\end{proof}

\section{Local statistics at multi-criticality}  \label{Section_local}

\subsection{Heuristics considerations} 
In this subsection, we present a heuristic way to guess $K= \lim\limits_{N \to \infty} K_N$,
\begin{equation} \label{K_N V lim}  
K(z,w)=
\begin{cases}
\hspace{3.5em}  2|zw| \, G(z^2,w^2) &\text{if} \quad 0< \tau < \tau_c, 
\vspace{0.5em}
\\
 |zw| \, G(z^2,w^2)  \erfc\Big( -\dfrac{z^2+\bar{w}^2}{\sqrt{2}}  \Big)  &\text{if} \quad \tau = \tau_c,
\vspace{0.5em}
\\
\hspace{6em} 0 &\text{otherwise,} 
\end{cases}
\end{equation}
 where $G$ is the \textit{Ginibre kernel}
\begin{equation}
G(z,w):=e^{ z\bar{w} -|z|^2/2-|w|^2/2  } .
\end{equation}

Let us first consider the random normal matrix model $\widetilde{\boldsymbol{\zeta}}=\{ \widetilde{\zeta}_j \}_{j=1}^N$ with fixed (i.e., independent of $N$) potential $\widetilde{Q}_{\alpha}$.
Note that by \eqref{density}, the macroscopic density $\delta$ at the origin is 
given by 
\begin{equation}
\delta=\Delta \widetilde{Q}_{\alpha} (0)=\frac{A^2}{4\alpha}, \quad \alpha>0.
\end{equation}
As usual, we rescale the normal eigenvalue ensemble via $\widetilde{z}_j=\sqrt{N\delta}\cdot \widetilde{\zeta}_j$, cf.\eqref{rescaling}, and write $\widetilde{K}_{N}$ for the associated correlation kernel.
Then, by combining the local bulk/edge universality shown in \cite{ameur2011fluctuations, hedenmalm2017planar} with Theorem~\ref{Thm_NWishart}, we have the convergence 
\begin{equation} \label{R_N Q conv}
\lim_{N \to \infty} \widetilde{K}_{N}(z,w) = 
\begin{cases}
\hspace{4em} G(z,w) &\text{if} \quad 0< \tau < \tau_c,  
\vspace{0.5em}
\\
G(z,w)\dfrac{1}{2} \erfc\Big(  -\dfrac{z+\bar{w}}{\sqrt{2}} \Big)  &\text{if} \quad \tau = \tau_c,  
\vspace{0.5em}
\\
\hspace{5em} 0 &\text{otherwise.}  
\end{cases}
\end{equation}

Now we consider the rescaled process $\widehat{\boldsymbol{z}}=\{ \widehat{z}_j \}_{j=1}^N$, ($\widehat{z}_j=\sqrt{N \delta} \cdot \widehat{\zeta}_j$) and write $\widehat{K}_{N}$ for the associated correlation kernel. Then by the asymptotic behaviour \eqref{Q_N asym}, one can expect that for $\alpha>0$, the kernel $\widehat{K}_{N}$ shares the same limit with \eqref{R_N Q conv}. However, we emphasise that this is not a direct consequence of \cite{ameur2011fluctuations,hedenmalm2017planar} due to the $N$-dependence of the potential. 

We now turn to the case of particle system \eqref{Gibbs square} in terms of
squared variables, with potential $V_N$. By \eqref{R_N V,Q rel}, we have the relation
\begin{align} \label{R V Q rel}
\begin{split}
K_{N}(z,w)&=\frac{1}{\sqrt{N \delta}}\bfK_{N}\Big( \frac{z}{ (N\delta)^{1/4}  } , \frac{w}{ (N\delta)^{1/4}  }  \Big)
\\
&=2|\zeta \eta|\frac{1}{\sqrt{N \delta}} \widehat{\bfK}_{N}\Big( \frac{z^2}{\sqrt{N\delta}} , \frac{w^2}{\sqrt{N\delta}} \Big)
\\
&=2|zw|\frac{1}{N \delta} \widehat{\bfK}_{N}\Big( \frac{z^2}{\sqrt{N\delta}}, \frac{w^2}{\sqrt{N\delta}} \Big)
=2 |zw| \, \widehat{K}_{N}(z^2,w^2).
\end{split}
\end{align}
Therefore by \eqref{R_N Q conv}, one can expect the convergence \eqref{K_N V lim}.

\begin{rem} \textbf{(Mass-one and Ward equations)}
For the limiting correlation kernel $K$ given in \eqref{K_N V lim}, let us define the \textit{Berezin kernel} 
\begin{equation}
B(z,w):=\frac{|K(z,w)|^2}{R(z)}.
\end{equation} 
Then, by \cite{ameur2015random,MR3975882} and the change of variables $z \mapsto z^2, w \mapsto w^2$, it is easy to show that the following form of mass-one and Ward's equation hold:
$$
\int_\C B(z,w) \, dA(w)=1; 
$$ 
$$
\bar{\partial}  \widetilde{C}(z)=2 R(z) -\Delta V_0(z)- \Delta \log R(z), \quad \widetilde{C}(z):=2z\int_{ \C } \frac{B(z,w)}{z^2-w^2} \, dA(w), 
$$
where $V_0(z):=|z|^4-2\log|z|.$ The derivation of such Ward equations for multi-fold interacting ensembles with general external potentials as well as the characterisation of radially symmetric solutions will appear in a forthcoming paper.
For the general form of Ward identities for Coulomb gases and their conformal field theoretical interpretation, we refer to 
\cite[Appendix 6]{MR3052311} and \cite{MR2240466}. 

\end{rem}

\subsection{The boundary Ginibre point field }

In this subsection we prove Theorem~\ref{Micro}. In the sequel, we focus on the critical regime
\begin{equation}
\tau=\tau_c=\frac{1}{\sqrt{1+\alpha_N}}, \quad \alpha_N=\frac{1-\tau^2}{\tau^2}.
\end{equation}
A function $c(z,w)$ is called \textit{cocycle} if there exists a continuous
unimodular function $g$ such that $c(z,w)=g(z)\overline{g(w)}$. Note that if $(c_N)_1^\infty$ is a sequence of cocycles, then $K_N$ and $c_N K_N$ defines same determinantal point process. Then, Theorem~\ref{Micro} follows from the following theorem and the relation \eqref{R V Q rel}.
\begin{thm} \label{Micro standard}
There exists a sequence of cocycles $(c_N)_{1}^\infty$ such that
\begin{equation}
\lim_{N \to \infty} c_N(z,w) \, \widehat{K}_N(z,w)=G(z,w) \dfrac12\erfc\Big( -\dfrac{z+\bar{w}}{\sqrt{2}} \Big)
\end{equation} 
uniformly for $z,w$ in compact subsets of $\C$.
\end{thm}
As a consequence of this theorem, the rescaled point process converges to the \textit{boundary Ginibre point field}, see \cite{MR2530159, forrester2010log}. 

To prove Theorem~\ref{Micro standard}, we need the following asymptotic behaviours of the modified Bessel functions $I_\nu$ and $K_\nu$.

\begin{lem} \label{Lem KI Ginibre}
As $\nu \to\infty$ through positive real values, we have 
\begin{equation}
I_\nu(\sqrt{\nu}z) \sim \frac{ 1}{\sqrt{2\pi}} \nu^{ -\frac{\nu+1}{2} } \Big( \frac{z}{2} \Big)^{\nu} e^{\nu} e^{ z^2/4 } \label{I nu asy}
\end{equation}
uniformly for $z$ in compact subsets of $\C$ and 
\begin{equation}
K_\nu(\sqrt{\nu}x) \sim \sqrt{ \frac{\pi}{2} } \nu^{ \frac{\nu-1}{2} } \Big( \frac{x}{2} \Big)^{-\nu} e^{-\nu} e^{ -x^2/4 } \label{K nu asy lem}
\end{equation}
uniformly for $x$ in compact subsets of $[0,\infty)$. In particular, we have 
\begin{equation} \label{K nu I nu G}
 \sqrt{ K_\nu (2\sqrt{\nu}|z|)  K_\nu (2\sqrt{\nu}|w|)  }  I_\nu (2\sqrt{\nu z \bar{w} } )  \sim \frac{1}{2\nu} \Big( \frac{z\bar{w}}{|zw|} \Big)^{ \frac{\nu}{2} } G(z,w)
\end{equation}
uniformly for $z,w$ in compact subsets of $\C$.
\end{lem}
\begin{proof}
The modified Bessel function $I_\nu$ has a well-known integral representation
\begin{equation}
I_\nu(z)=\frac{(z/2)^\nu }{ \sqrt{\pi} \, \Gamma(\nu+\frac12 ) } \int_{-1}^{1} (1-t^2)^{\nu-\frac12 } e^{\pm z t} \, dt, \quad \Big(\re \nu > -\frac12\Big),
\end{equation}
see e.g., \cite[Eq.(10.32.2)]{olver2010nist}.
Using this representation and Stirling's formula, we have
$$
I_\nu(\sqrt{\nu}z) \sim \frac{ 1 }{\sqrt{2}\pi} \nu^{-\frac{\nu+1}{2}}  \Big( \frac{z}{2} \Big)^\nu  e^\nu \int_{ -\sqrt{\nu} }^{ \sqrt{\nu} } (1-t^2/\nu)^{\nu-\frac12} e^{z t} \, dt.
$$
For any sequence $(\nu_k)_{1}^\infty$ of positive numbers such that $\nu_k \to \infty$, let us consider a sequence of the probability densities 
$$
p_k(t):=\frac{\mathbbm{1}_{[-\sqrt{\nu_k},\sqrt{\nu_k}]  } (1-t^2/\nu_k)^{\nu_k-\frac12}  }{  \int_{-\sqrt{\nu_k}}^{\sqrt{\nu_k}} (1-t^2/\nu_k)^{\nu_k-\frac12} \, dt    }=  \sqrt{ \frac{\nu_k}{\pi} }\, \frac{\mathbbm{1}_{[-\sqrt{\nu_k},\sqrt{\nu_k}]  } (1-t^2/\nu_k)^{\nu_k-\frac12}  }{   \Gamma(\nu_k+\frac12)/\Gamma(\nu_k)  }
$$
and write $p(t)=\frac{1}{\sqrt{\pi}} e^{-t^2}.$ By Gautschi's inequality: $\Gamma(x+1) < \sqrt{x+1} \, \Gamma(x+\frac12)$ $(x>0)$ and a fundamental inequality on logarithm: $\log (1+x) \le x$  $(x > -1),$ we obtain that for $t \in (-\sqrt{\nu_k},\sqrt{\nu_k})$ and a sufficiently large $k,$ 
\begin{align*}
p_k(t) &< \frac{1}{\sqrt{\pi}} \sqrt{\frac{\nu_k+1}{\nu_k}} (1-t^2/\nu_k)^{\nu_k-\frac12} < \sqrt{\frac{\nu_k+1}{\nu_k}} e^{  \frac{t^2}{2\nu_k} } p(t) <  2 \, p(t).
\end{align*}
Set $f_k(t)=p_k(t) e^{zt}$. Then we have  $|f_k(t)| \le 2\,p(t) e^{ (\re z)t}$ and $ f_k(t) \to p(t) e^{zt}$. We now use Lebesgue's dominated convergence theorem to derive 
\begin{equation} 
\int_{ -\sqrt{\nu} }^{ \sqrt{\nu} } (1-t^2/\nu)^{\nu-\frac12} e^{z t} \, dt \sim \int_{ -\infty }^\infty e^{-t^2} e^{zt} \, dt= e^{z^2/4} \sqrt{\pi}
\end{equation}
uniformly for $z$ in compact subsets of $\C$.
This completes the proof of \eqref{I nu asy}.

On the other hand, the modified Bessel function $K_\nu$ has an integral representation
$$
K_\nu(az)=\frac{\Gamma(\nu+\frac12)}{\sqrt{\pi}}\Big( \frac{2z}{a} \Big)^\nu \int_0^\infty \frac{\cos a t \, dt}{(t^2+z^2)^{\nu+\frac12}  }, \quad \Big(\re \nu>-\frac12, \re z>0, a>0\Big),
$$ 
see e.g., \cite[Eq.(10.32.11)]{olver2010nist}. Then, similar to the above computation together with this representation, we have 
\begin{align*}
K_\nu(\sqrt{\nu}x)&=\frac{\Gamma(\nu+\frac12)}{2\sqrt{\pi}}\nu^{ -\frac{\nu+1}{2} }\Big( \frac{x}{2} \Big)^{-\nu} \frac{1}{x} \int_{-\infty}^\infty  \Big( 1+\frac{t^2}{\nu x^2} \Big)^{-\nu-\frac12} \cos t \, dt
\\
&\sim \frac{1}{\sqrt{2}} \, \nu^{ \frac{\nu-1}{2} } \Big( \frac{x}{2} \Big)^{-\nu} e^{-\nu} \frac{1}{x} \int_{-\infty}^\infty  e^{-t^2/x^2} \cos t\, dt
\end{align*}
uniformly for $x$ in compact subsets of $(0,\infty).$
Here 
$$
\int_{-\infty}^\infty  e^{-t^2/x^2} \cos t \, dt=e^{-x^2/4} \re  \int_{-\infty}^\infty e^{-(t/x-ix/2)^2} \, dt=x \, e^{-x^2/4} \sqrt{\pi},
$$
which leads to \eqref{K nu asy lem}. 

For the second assertion, notice that 
$$
I_\nu(\sqrt{\nu} x)  K_\nu(\sqrt{\nu} x) \sim \frac{1}{2\nu}, \quad (0<x<\infty).
$$
Then by \eqref{I nu asy}, we have
\begin{align}
& \quad 2\nu \sqrt{ K_\nu (2\sqrt{\nu}|z|)  K_\nu (2\sqrt{\nu}|w|)  }  I_\nu (2\sqrt{\nu z \bar{w} } ) 
\\
&\sim   \frac{  I_\nu (2\sqrt{\nu z \bar{w} } )   }{  \sqrt{ I_\nu (2\sqrt{\nu}|z|) I_\nu (2\sqrt{\nu}|w|)  }  } \sim \Big( \frac{z\bar{w}}{|zw|} \Big)^{ \frac{\nu}{2} } G(z,w),
\end{align}
which completes the proof. 
\end{proof}

Now we start to investigate the structure of the correlation kernel $\widehat{K}_N$. In the sequel, we write 
$$a=\frac{1-\tau^2}{\tau} \sqrt{\nu}.$$ 
\begin{lem}
	For any integer $N \ge 1$ and the parameter $\tau:=B/A \in (0,1)$, we have
	\begin{align}  \label{K_N inner sum}
	\begin{split}
	\widehat{K}_N(z,w)
	&=  2 \, \nu^{\frac{\nu}{2}+1}     \sqrt{ K_\nu ( 2 \sqrt{\nu} |z| )  K_\nu ( 2 \sqrt{\nu} |w| ) }  
	\\
	&\times  |z w|^{\frac{\nu}{2}  }  \sum_{k=0}^{N-1}  \frac{  \nu^k (z \bar{w} )^k  }{k! \, \Gamma(k+\nu+1)} \mathcal{I}_k(z+\bar{w}),
	\end{split}
	\end{align}
	where 
	$$
	\mathcal{I}_k(z):=(1-\tau^2)^{\nu+2k+1} e^{\tau\sqrt{\nu}z } \sum_{j=0}^{N-1-k}  \tau^{2j} L_{j}^{\nu+2k} ( az ).
	$$
	Furthermore, 
	\begin{equation}  \label{RN 1_k}
		\mathcal{I}_k(z)=\frac{1}{2\pi i} \oint_{ \gamma } \Big( \frac{1-s}{1-\tau^2} \Big)^{-\nu-2k-1} e^{ \tau \sqrt{\nu}z (1-\frac{s}{1-s}\frac{1-\tau^2}{\tau^2}) }    \frac{ 1-(\tau^2/s)^{N-k} }{s-\tau^2} \, ds,
	\end{equation}
	where the integration contour $\gamma$ encircles the origin $s=0$ but not the point $s=1$.
\end{lem}
\begin{proof}
Using orthogonality relation \eqref{Orthogonal relation}, we first express $\widehat{K}_N$ (up to a cocycle) in terms of $K_\nu$ and $L_j^\nu$ $(0 \le j \le N-1)$ as follows:
\begin{align*} 
\begin{split}
\widehat{K}_N(z,w)
&= 2 \, \nu^{\frac{\nu}{2}+1}    ( 1-\tau^2  )^{\nu+1}  \sqrt{  K_\nu ( 2 \sqrt{\nu} |z| )  K_\nu ( 2 \sqrt{\nu} |w| ) } 
\\
& \times |z w|^{ \frac{\nu}{2} } e^{ \tau \sqrt{\nu} (z+\bar{w})  } \sum_{j=0}^{N-1} \frac{\tau^{2j}j!}{\Gamma(j+\nu+1)}   L_j^\nu( az ) \overline{ L_j^\nu ( a w ) }.
\end{split}
\end{align*}
Using the formula (see \cite[Eq.(10.12.42)]{MR698780})
\begin{equation} \label{Lag addition}
\frac{j!}{\Gamma(j+\nu+1)}L_j^\nu(z) \overline{L_j^\nu(w)}=\sum_{k=0}^{j} \frac{(z\bar{w})^{k}}{k! \, \Gamma(k+\nu+1)} L_{j-k}^{\nu+2k} (z+\bar{w})
\end{equation}
and changing the order of summations, we obtain 
$$
\sum_{j=0}^{N-1} \frac{\tau^{2j}j!}{\Gamma(j+\nu+1)}L_j^\nu(z) \overline{L_j^\nu(w)}=\sum_{k=0}^{N-1}  \frac{\tau^{2k}(z\bar{w})^{k}}{k! \, \Gamma(k+\nu+1)} \sum_{j=0}^{N-1-k} \tau^{2j} L_{j}^{\nu+2k} (z+\bar{w}).
$$

Next, we show \eqref{RN 1_k}.
For this, recall the contour integral representation of the Laguerre polynomials
\begin{equation}
L_j^\nu(x)= \frac{1}{2\pi i} \oint_{ \gamma } \frac{  e^{ -\frac{s x}{1-s} }  }{(1-s)^{\nu+1} s^{j+1} } \, ds,
\end{equation}
where the contour $\gamma$ encircles the origin $s=0$ but not the point $s=1$.
From this expression, we obtain
$$
\sum_{j=0}^{N-1-k} \tau^{2j} L_{j}^{\nu+2k} (a x) 
=\frac{1}{2\pi i} \oint_{ \gamma }  \frac{  e^{ -\frac{s a x}{1-s} }  }{(1-s)^{\nu+2k+1}  }  \frac{ 1-(\tau^2/s)^{N-k} }{s-\tau^2} \, ds,
$$
which completes the proof. 
\end{proof}

Later, we need the uniform boundedness of $\mathcal{I}_k|_{\R_{-}}$. 
\begin{lem} \label{Lem RN1 bdd}
For any $k $ and $x \in (-\infty,0)$, we have
\begin{equation} \label{RN1 bdd}
0 \le \mathcal{I}_k(x) \le 1.
\end{equation}	
\end{lem}
\begin{proof}
It is obvious that $L_j^\nu(x)>0$ for $x<0$. On the other hand, by the generating function of the Laguerre polynomial (see e.g., \cite[Eq.(5.1.9)]{MR0372517}), we have 
\begin{equation}
\sum_{j=0}^{\infty}  \tau^{2j} L_{j}^{\nu+2k} (  ax )=(1-\tau^2)^{-\nu-2k-1} e^{-\tau\sqrt{\nu}x },
\end{equation}
which leads to \eqref{RN1 bdd}. 
\end{proof}

The next lemma provides the asymptotic behaviour of $\mathcal{I}_k$. 
\begin{lem} \label{Lem RN 1 asymp}
For a given value of $k=o(\sqrt{N})$, we have
\begin{equation}
\mathcal{I}_k(z)= \dfrac12\erfc\Big( -\dfrac{z}{\sqrt{2}} \Big)+O\Big( \frac{k}{\sqrt{N}} \Big).
\end{equation}
\end{lem}

\begin{proof}
Let us choose the integration contour $\gamma$ in \eqref{RN 1_k} such that it does not encircle the point $s=\tau^2$. 
Then, we have 
\begin{equation}  \label{RN 1_k f}
\mathcal{I}_k(z)
=\frac{1}{2\pi i} \oint_{ \gamma } \Big( \frac{1-s}{1-\tau^2} \Big)^{-2k-1} \Big(\frac{s}{\tau^2} \Big)^{k} e^{ -N h(s)  }  e^{ \tau \sqrt{\nu}z(1-\frac{s}{1-s}\frac{1-\tau^2}{\tau^2}) } \frac{ds}{\tau^2-s},
\end{equation}
where 
\begin{equation}
h(s):=\alpha_N \log \frac{1-s}{1-\tau^2}+ \log \frac{s}{\tau^2}.
\end{equation} 

We compute the asymptotic behaviour of the integral \eqref{RN 1_k f} by means of the standard saddle point analysis, see \cite[Section 4.6]{MR1429619}. We also refer to \cite[Section V]{akemann2010interpolation} (resp., \cite[Section 4]{MR2594353}) for a similar analysis for the edge kernel of orthogonal Laguerre polynomials with fixed $\nu$, (resp., Hermite polynomials) in the complex plane. We remark that the computations presented here are significantly simpler than those in \cite{akemann2010interpolation,MR2594353}, where certain double contour integrals were analysed. This simplification essentially comes from the formula \eqref{Lag addition}.

The saddle point analysis relies on the fact that the main contribution of the contour integral comes from the critical point of the function $h$. Therefore we deform the contour $\gamma$ so that it passes through a neighbourhood of the critical point with certain direction, which leads to a real integral representation for the main contribution of the contour integral \eqref{RN 1_k f}.   

Note that the function $h$ has the unique critical point $s_*$,
\begin{equation}
s_*=\frac{1}{1+\alpha_N}=\tau^2. 
\end{equation}
Furthermore, $h(s_*)=0$ and 
\begin{equation}
h''(s_*)=-\frac{1}{\tau^4 (1-\tau^2)}.
\end{equation}	
Deforming the integration contour $\gamma$, let us assume that $\gamma$ vertically passes through the real axis near the point $\tau^2$. Let us parametrise 
\begin{equation}
\label{param}
s=s_*+\frac{ iy-\eps }{ \sqrt{ -Nh''(s_*) }}, \quad \eps>0.
\end{equation}
Here, $\eps$ is a regularisation parameter which will be taken $\eps \downarrow 0$. 
We now analyse the asymptotic behaviour of each term in \eqref{RN 1_k f}. We have
$$
	\Big( \frac{1-s}{1-\tau^2} \Big)^{-2k-1} \Big(\frac{s}{\tau^2} \Big)^{k}
= 1+ \frac{1+\tau^2}{\sqrt{1-\tau^2}}  \frac{k}{\sqrt{N}} \,(iy-\eps)  +O\Big(  \frac{k^2}{N}  \Big)
$$ 
and
	\begin{equation}
	e^{-N h(s)}=e^{ (iy-\eps)^2/2+O(1/\sqrt{N}  ) }.
	\end{equation}
	Next, we have
	$$
		\tau \sqrt{\nu}z \, \Big(1-\frac{s}{1-s}\frac{1-\tau^2}{\tau^2}\Big)
	= -z(iy-\eps)+O\Big(\frac{1}{\sqrt{N}} \Big).
	$$
Due to the parametrisation \eqref{param}, we have	
\begin{equation}
\frac{1}{2\pi i}\frac{ds}{\tau^2-s}=-\frac{1}{2\pi}\frac{ dy}{iy-\eps}.
\end{equation}
	
	Combining all above, we compute the leading term of the integration \eqref{RN 1_k f} as
	\begin{align*}
	&\quad \frac{1}{2\pi} \lim_{\eps \downarrow 0} \lim_{R \to \infty}\int_{ -R }^R e^{ (iy-\eps)^2/2-z(iy-\eps) } \frac{ dy}{\eps-iy}
	\\
	&=  \frac{i}{2\pi} e^{ -z^2/2 } \lim_{\eps \downarrow 0}    \lim_{R \to \infty}\int_{ -R+iz }^{R+iz } e^{ -\zeta^2/2-i \eps (\zeta-iz) } \frac{d\zeta}{\zeta-i(z- \eps) }, 
	\end{align*}
	where $\zeta=y+iz$. Applying Cauchy's integral formula to a holomorphic function 
	$\zeta \mapsto  e^{-\zeta^2/2-i \eps (\zeta-iz)}$ along the positively oriented sides of a “large” parallelogram with vertices $-R,R,R+iz,-R+iz$,
	we obtain
	\begin{align*}
2\pi i\, e^{ (z-\eps)^2/2-\eps^2 }	 &= \lim_{R \to \infty}\int_{ -R }^{R } e^{ -\zeta^2/2-i \eps (\zeta-iz) } \frac{d\zeta}{\zeta-i(z- \eps) }  
\\
&-\lim_{R \to \infty}\int_{ -R+iz }^{R+iz } e^{ -\zeta^2/2-i \eps (\zeta-iz) } \frac{d\zeta}{\zeta-i(z- \eps) } .
	\end{align*}
It is easy to check that the integrals along the other sides of the parallelogram converge to $0$ as $R \to \infty$. 
Changing the order of limits, it follows from the above two equations that 
	\begin{align*}
	& \quad \frac{1}{2\pi}  \lim_{R \to \infty} \lim_{\eps \downarrow 0} \int_{ -R }^R e^{ (iy-\eps)^2/2-z(iy-\eps) } \frac{ dy}{\eps-iy}
		\\
	&=  1+\frac{i}{2\pi} e^{-z^2/2 }  \lim_{R \to \infty} \int_0^R e^{-\zeta^2/2 } \Big( \frac{1}{\zeta-iz}-\frac{1}{\zeta+iz} \Big) \, d\zeta.
	\end{align*}
    Thus by \cite[Eq.(7.7.1)]{olver2010nist}, we conclude that 
	\begin{align}
	\begin{split}
	\mathcal{I}_k(z) &\sim 1-\frac{z}{\pi} e^{-z^2/2  } \int_{0}^{\infty} e^{ -\zeta^2/2 } \frac{d\zeta}{\zeta^2+z^2}
	\\
	&= 1-\frac12 \erfc\Big( \frac{z}{\sqrt{2}} \Big)= \frac12 \erfc\Big( -\frac{z}{\sqrt{2}} \Big).
	\end{split}
	\end{align}
	It is left to the reader as an exercise to analyse the subleading term in \eqref{RN 1_k f}.
\end{proof} 

We are now ready to prove Theorem~\ref{Micro standard}. Due to the relation \eqref{R V Q rel}, Theorem~\ref{Micro} follows immediately from Theorem~\ref{Micro standard}.
\begin{proof}[Proof of Theorem~\ref{Micro standard}]
Fix small $\eps>0$ and let $m=\lfloor N^{\frac14-\eps} \rfloor$. 
\medskip
\\
\noindent \textit{Claim 1.} Only the first $m$ terms in the summation of \eqref{K_N inner sum} contribute to $\widehat{K}_N$, i.e.,
\begin{align}  \label{claim 1}
\begin{split}
\widehat{K}_N(z,w)
&\sim 2 \, \nu^{\frac{\nu}{2}+1}     \sqrt{ K_\nu ( 2 \sqrt{\nu} |z| )  K_\nu ( 2 \sqrt{\nu} |w| ) }  
\\
&\times  |z w|^{\frac{\nu}{2}  }  \sum_{k=0}^{m-1}  \frac{  \nu^k (z \bar{w} )^k  }{k! \, \Gamma(k+\nu+1)} \mathcal{I}_k(z+\bar{w}).
\end{split}
\end{align}
Then, by Lemma~\ref{Lem RN 1 asymp} and with the choice of $m$, we have 
\begin{align}  \label{K hat asym}
\begin{split}
\widehat{K}_N(z,w)
&\sim  \nu^{\frac{\nu}{2}+1}     \sqrt{ K_\nu ( 2 \sqrt{\nu} |z| )  K_\nu ( 2 \sqrt{\nu} |w| ) }  
\\
& \times  |z w|^{\frac{\nu}{2}  } \erfc\Big( -\frac{z+\bar{w}}{\sqrt{2}} \Big) \sum_{k=0}^{m-1}  \frac{  \nu^k (z \bar{w} )^k  }{k! \, \Gamma(k+\nu+1)}.
\end{split}
\end{align}
\noindent \textit{Claim 2.} We have
\begin{equation}  \label{claim 2}
\sum_{k=0}^{m-1}  \frac{  \nu^k (z \bar{w} )^k  }{k! \, \Gamma(k+\nu+1)} 
\sim \sum_{k=0}^{\infty}  \frac{  \nu^k (z\bar{w})^k }{k! \, \Gamma(k+\nu+1)} =  \nu^{ -\frac{\nu}{2} } (z \bar{w})^{-\frac{\nu}{2}}  I_\nu( 2\sqrt{\nu z \bar{w}}  ).
\end{equation}
It follows from \eqref{K hat asym} and Claim 2 that 
\begin{align}  
\begin{split}
\widehat{K}_N(z,w)
&\sim  \Big( \frac{ z\bar{w} }{|z w|} \Big)^{-\frac{\nu}{2}  } \frac12 \erfc\Big( -\frac{z+\bar{w}}{\sqrt{2}} \Big) 
\\
&\times 2 \nu    \sqrt{ K_\nu ( 2 \sqrt{\nu} |z|)  K_\nu ( 2 \sqrt{\nu} |w| ) }    I_\nu( 2\sqrt{\nu z \bar{w}}  ).
\end{split}
\end{align}
Now Theorem~\ref{Micro standard} follows from Lemma~\ref{Lem KI Ginibre}. 

We first prove Claim 2. Note that the identity in \eqref{claim 2} follows from \cite[Eq.(10.25.2)]{olver2010nist}. Thus to obtain \eqref{claim 2}, it suffices to show that for $M>0,$
\begin{equation} \label{claim 2-1}
\Gamma(\nu+1) \sum_{k=m}^{\infty} \frac{(\nu M)^k}{k! \, \Gamma(k+\nu+1)}=o(1).
\end{equation}
By the well-known estimate of the Gamma function 
\begin{equation} \label{Gamma estimate}
1< \frac{1}{\sqrt{2\pi}} \frac{e^x}{ x^{x+\frac12} } \Gamma(x+1) < e^{ \frac{1}{12x} },
\end{equation}
we have 
\begin{align}
\begin{split}
\Gamma(\nu+1)\sum_{k=m}^{\infty}  \frac{  (\nu M)^k }{k! \, \Gamma(k+\nu+1)} 
& \le \frac{\Gamma(\nu+1)}{\sqrt{2\pi}} \sum_{k=m}^{\infty} \frac{1}{k!} \frac{e^{ k+\nu } (\nu M)^{k}   }{ (k+\nu)^{ k+\nu+\frac12 }  } 
\\
& \le  \frac{\Gamma(\nu+1)}{\sqrt{2\pi}}  \frac{e^ \nu }{ (m+\nu)^{ \nu+\frac12 }   } \sum_{k=m}^{\infty} \frac{1}{k!}  \Big(  \frac{e \nu M  }{ m+\nu }  \Big)^k
\\
& \le e^{-m+O(1)}  \sum_{k=m}^{ \infty } \frac{1}{k!}  \Big(  \frac{e \nu M  }{ m+\nu }  \Big)^k,
\end{split}
\end{align}
which implies the desired asymptotic behaviour \eqref{claim 2-1}.

Next we prove Claim 1. Notice that by Lemma~\ref{Lem RN 1 asymp}, it suffices to show that 
\begin{equation}
\Gamma(\nu+1) \sum_{k=m}^{N-1}  \frac{  (\nu M)^k  }{k! \, \Gamma(k+\nu+1)} \, \mathcal{I}_k(z+\bar{w})=o(1).
\end{equation}
Note that by \eqref{Laguerre poly}, we have $|L_j^\nu(z)| \le L_j^\nu(-|z|).$
Therefore due to Lemma~\ref{Lem RN1 bdd}, we obtain
\begin{align}
\begin{split}
|\mathcal{I}_k(z)|&=(1-\tau^2)^{\nu+2k+1} e^{\tau\sqrt{\nu} \re z } \, \Big| \sum_{j=0}^{N-1-k}  \tau^{2j} L_{j}^{\nu+2k} (  az ) \, \Big| 
\\
& \le  (1-\tau^2)^{\nu+2k+1} e^{\tau\sqrt{\nu} \re z } \sum_{j=0}^{N-1-k}  \tau^{2j} L_{j}^{\nu+2k} (  -a|z| )
\\
& = e^{ \tau\sqrt{\nu} (|z|+\re z)  } \, \mathcal{I}_k(-|z|) \le e^{ \tau\sqrt{\nu} (|z|+\re z)  }.
\end{split}
\end{align}
Thus by Lemma~\ref{Lem RN 1 asymp}, for any $z$ contained in a given compact set, there exist positive constants $C_1,C_2$ such that 
\begin{equation} \label{RN1 bdd gen}
|\mathcal{I}_k(z)| \le 
\begin{cases}
C_1 &\text{if  } k=o(\sqrt{N}),
\\
e^{C_2 \sqrt{N}}  &\text{for all  } k.
\end{cases}
\end{equation}
Similarly we utilize the estimate \eqref{Gamma estimate} to obtain
\begin{align*}
&\quad \Gamma(\nu+1)\sum_{k=m}^{N-1}  \frac{  (\nu M)^k }{k! \, \Gamma(k+\nu+1)} \, |\mathcal{I}_k(z+\bar{w})|
\\
&\le  \frac{\Gamma(\nu+1)}{\sqrt{2\pi}}  \frac{e^ \nu }{ (m+\nu)^{ \nu+\frac12 }   } \sum_{k=m}^{N-1} \frac{\lambda^k}{k!}  \, |\mathcal{I}_k(z+\bar{w})|
\\
&\le  (1+o(1)) \, e^{-m} \sum_{k=m}^{N-1} \frac{\lambda^k}{k!}    \, |\mathcal{I}_k(z+\bar{w})|,
\end{align*}
where $\lambda:=e\nu M/(m+\nu).$ 
Now let us choose a constant $\theta \in (1/4,1/2).$ Then by \eqref{RN1 bdd gen}, we have 
\begin{align}
\sum_{k=m}^{N-1} \frac{\lambda^k}{k!}  \, |\mathcal{I}_k(z+\bar{w})|
&=\sum_{k=m}^{ \lfloor N^{\theta } \rfloor }  \frac{\lambda^k}{k!} \, |\mathcal{I}_k(z+\bar{w})|+\sum_{k= \lfloor N^\theta \rfloor+1 }^{N-1} \frac{\lambda^k}{k!} \,  |\mathcal{I}_k(z+\bar{w})|
\\
&\le C_1\sum_{k=m}^{ \lfloor N^{\theta } \rfloor } \frac{\lambda^k}{k!}   + e^{C_2 \sqrt{N}} \sum_{k= \lfloor N^{\theta } \rfloor+1 }^{N-1} \frac{\lambda^k}{k!}  
\\
&\le  e^\lambda \, \Big( C_1 \mathbb{P}[X \ge m]  + e^{C_2 \sqrt{N}} \mathbb{P}[X \ge  N^{\theta } ] \Big), 
\end{align}
 where $X$ is a Poisson random variable with intensity $\lambda.$ Using the normal approximation of the Poisson distribution, there exist positive constants $c_1,c_2$ such that 
\begin{equation}
\sum_{k=m}^{N-1} \frac{\lambda^k}{k!} \,  |\mathcal{I}_k(z+\bar{w})| \le e^\lambda \Big( C_1 e^{-c_1 m^2}  + e^{C_2 \sqrt{N}} e^{-c_2 N^{ 2\theta } } \Big)=o(1),
\end{equation}
which completes the proof. 
\end{proof}

\noindent \textbf{Acknowledgements.} The authors gratefully acknowledge Seung-Yeop Lee for several helpful comments concerning the proof of Theorem~\ref{Thm_NWishart}.
The present work was initiated when the authors visited the Department of Mathematics at KTH Royal Institute of Technology, and we wish to express our gratitude to Maurice Duits, H\aa kan Hedenmalm, and Kurt Johansson for the invitation and hospitality.

\bibliographystyle{abbrv}
\bibliography{RMTbib}

\end{document}